\documentclass[1p,times]{elsarticle}
\usepackage{amsmath,amsthm,amstext,amscd,amssymb,euscript,url}
\usepackage{mathrsfs,euscript,dsfont}
\usepackage{mathtools}
\usepackage{enumerate}
\usepackage{bbm}
\usepackage{mathtools}
\usepackage{etoolbox}

\DeclareFontFamily{U}{matha}{\hyphenchar\font45}
\DeclareFontShape{U}{matha}{m}{n}{
<-6> matha5 <6-7> matha6 <7-8> matha7
<8-9> matha8 <9-10> matha9
<10-12> matha10 <12-> matha12
}{}
\DeclareSymbolFont{matha}{U}{matha}{m}{n}

\DeclareFontFamily{U}{mathx}{\hyphenchar\font45}
\DeclareFontShape{U}{mathx}{m}{n}{
<-6> mathx5 <6-7> mathx6 <7-8> mathx7
<8-9> mathx8 <9-10> mathx9
<10-12> mathx10 <12-> mathx12
}{}
\DeclareSymbolFont{mathx}{U}{mathx}{m}{n}

\DeclareMathDelimiter{\vvvert} {0}{matha}{"7E}{mathx}{"17}%

\DeclarePairedDelimiterX{\normiii}[1]
{\vvvert}
{\vvvert}
{\ifblank{#1}{\:\cdot\:}{#1}}

\usepackage{xcolor}

\usepackage{ushort} 

\def\1{{\mathds 1}}

\newtheorem{theorem}{Theorem}[section]
\newtheorem{lemma}[theorem]{Lemma}
\newtheorem{corollary}[theorem]{Corollary}
\newtheorem{proposition}[theorem]{Proposition}
\newtheorem{remark}[theorem]{Remark}

\newcommand{\Z}{\mathds Z}
\newcommand{\R}{\mathds R}

\newcommand{\D}{\mathbb{D}}

\newcommand{\vM}{{\underline {\mathcal M}^{(\sK)}}}
\newcommand{\vsbullet}{\,\sbullet\,}

\DeclareMathOperator{\e}{\mathrm{e}}

\newcommand{\un}{\mathbbm{1}}

\newcommand{\sK}{{\scriptscriptstyle K}}
\newcommand{\TK}{T^{(\sK)}}

\newcommand{\proba}{{{\mathds P}}^{(\sK)}}
\newcommand{\QK}{\widehat { \mathds P}^{(\sK)}}
\newcommand{\EQK}{\widehat { \mathds E}^{(\sK)}}

\newcommand{\esperance}{{\mathds E}^{(K)}}
\newcommand{\nf}{\underline{n}_{*}^{(\sK)}}
\newcommand{\xf}{\underline x_{*}}
\newcommand{\lambK}{\lambda_{\sK}}
\newcommand{\laKqsd}{\nu_{\sK}}
\newcommand{\lauK}{u_{\sK}}
\newcommand{\NK}{N^{(K)}}
\newcommand{\sgK}{P^{(\sK)}}
\newcommand{\vn}{\underline n}
\newcommand{\vm}{\underline m}
\newcommand{\vx}{\underline x}
\newcommand{\vy}{\underline y}
\newcommand{\vz}{\underline z}
\newcommand{\vun}{\underline 1}
\newcommand{\vr}{\vec r}
\newcommand{\vs}{\underline s}
\newcommand{\vX}{\underline X}

\newcommand{\vN}{{\underline N^{(\sK)}}}
\newcommand{\vB}{\underline B}
\newcommand{\vD}{\underline D}
\newcommand{\ve}{\underline e}
\newcommand{\Oun}{\mathcal{O}(1)}
\newcommand{\etaK}{\eta_{\sK}}
\newcommand{\deltaK}{\delta_{\sK}}
\newcommand{\MK}{M_{\sK}}
\newcommand{\JK}{J_{\sK}}
\newcommand{\AK}{A_{\sK}}
\newcommand{\qsd}{q.s.d. }
\newcommand{\vecz}{{\ushort{0}}}
\newcommand{\tauex}{\tau^{\sK}_{\vecz}}
\newcommand{\moyenne}{\mathscr{M}(A,\,\alpha,\,K)}

\newcommand{\domq}{\Z^d_{+}\backslash\{\vecz\}}
\newcommand{\rpd}{\R_{+}^{d}}
\newcommand{\sbullet}{\raisebox{0.4mm}{\scalebox{0.6}{\textbullet}}}

\newcounter{numquestion}
\addtocounter{numquestion}{1}

\begin{document}

\begin{frontmatter}
\journal{Stochastic Processes and their Applications}

\title{Threatening  excursions in large population\\
 quasi-stationary birth and death  systems.\\
On a question of Antonio Galves.}


\author[1]{Pierre Collet \corref{cor1}}
\ead{pierre.collet@polytechnique.edu}
\author[2]{Servet Mart\'{\i}nez}
\ead{smartine@dim.uchile.cl}
\author[3]{Sylvie M\'el\'eard}
\ead{sylvie.meleard@polytechnique.edu}

\affiliation[1]{organization={Centre de Physique Th\'eorique, CNRS, Ecole
polytechnique, Institut Polytechnique de Paris},
city={Palaiseau},
 country={France}} 
\affiliation[2]{organization={Departamento de Ingenier\'{\i}a Matem\'atica
and Centro de Modelamiento Matem\'atico},
addressline={Casilla 170-3 Correo 3}, 
city={Santiago}, 
 country={Chile}}
\affiliation[3]{organization={CMAP, Ecole Polytechnique, CNRS, Institut
    polytechnique de Paris, Inria}, 
city={Palaiseau},
country={France}}

\cortext[cor1]{Corresponding author}


\begin{abstract}
  We consider time continuous multispecies birth and death processes
  in a regime of large populations.  The jump rates depend on a large
  scaling parameter $K$ modeling the charge capacity. When $K$ tends
  to infinity, the process is close (in finite time) to a dynamical
  system containing a non zero global attracting equilibrium and zero
  as unstable equilibrium.  For each fixed $K$, extinction in finite
  time occurs almost surely and a quasi-stationary distribution occurs
  naturally in the study of the statistics over times scales which are
  large but smaller than the extinction time scale.  Before this
  catastrophic event the process makes many unsuccessful large
  deviations attempts with time scales corresponding to how far it
  deviates from the quasi-equilibrium. The paper concerns the
  statistical description of these typical trajectories  starting  from the
  quasi-stationary distribution until extinction.  An unusual mixing
  property  yields  large time scale behavior for the
  process starting  from a fixed state.  We give a precise statistical description of the
  successive exit  times  of the process rescaled by $K$ from a
  neighborhood of the equilibrium of the dynamical system in a
  clumping time scale and prove their asymptotic Poisson
  distribution. We also  give a precise description of the asymptotic  distribution of the successive records until
  extinction.
\end{abstract}


\end{frontmatter}

\noindent\textsl{This article is dedicated to the memory of Antonio Galves
  (1947-2023). Antonio was  an eminent colleague and also a very
  close friend. We will always miss his brilliant scientific
  competence and extensive culture, his inexhaustible kindness, energy
  and optimism,  
 and his charming smile.}




\section{Introduction}\label{intro}

\subsection{The model and main results}

We consider  time continuous multispecies birth and death processes  in a regime of
large populations.  For $d$ species ($d\ge1$) the state of the system
at  a given time is
 an integer valued vector in $\Z^d_{+}$ whose components
are the number of individuals of each of the $d$ species.
 The jump rates depend on a  large scaling parameter $K$  of
 the 
 charge capacity. At state $\vn\in \Z^d_{+}$, the population birth rate 
  is given by $K\,\vB(\vn/K)$ and the death rate by
  $K\,\vD(\vn/K)$, where $\vB$ and $\vD$ are regular vector fields
  in $\rpd$ with nonnegative components. 
  
  We will denote by $\vX$ the vector field $\vB-\vD$ and by 
$(\varphi_{s})_{s\in \R_{+}}$ the associated semi-flow. Recall that $\varphi_{0}= identity$. 
  
  We will impose
  later on some additional hypotheses on $\vB$ and $\vD$ following
  \cite{CCM1}, \cite{CCM2} and \cite{CCMM}. In particular, $\vB(\vecz)=\vD(\vecz)=0$.
The process is denoted by $\big(\vN(t)\big)_{t\ge0}$. Under the
hypotheses in  \cite{CCM1}, \cite{CCM2} and \cite{CCMM} (see Subsection \ref{hypo}), the process
will reach almost surely in finite time the state $\vecz$ (total extinction). In the absence of
spontaneous generation, this state is absorbing and the corresponding
Dirac mass is the unique invariant (ergodic) probability measure. 

These assumptions also  imply that there exists  a globaly attracting equilibrium $\xf$ belonging to the interior of $\rpd$ \textup{(}fixed point of $\vX$
\textup{)} such that
\[
\vB(\xf)-\vD(\xf)=\vX(\xf)= {0}\,.
\]

It is a classical result (see e.g. \cite{Kurtz}) that
if the process starts from $[K\vx_{0}]$ ($\vx_{0}\in \rpd\backslash \{\vecz\})$ and  $\,[\;]$ denotes the
vector of integer parts), then the rescaled process $\vN(\sbullet)/K$
 converges,  as $K$ tends to infinity on any finite time interval, to
the solution of the differential system
$$\frac{d \vx}{dt} = \vB(\vx)-\vD(\vx)=\vX(\vx).$$

We now describe briefly  the behavior of a  trajectory of
the process up to extinction.

Let $\vx_{0}\in\rpd\backslash \{\vecz\}$ and consider a typical trajectory starting at
$[K\vx_{0}]$. With a probability very close to one the process will
first reach a small neighborhood of 
$$
\nf=[K\,\xf]
$$
and fluctuate around this point.  
In \cite{CCMM} we studied  the fluctuations of the
process around the point $\nf$. These typical
fluctuations are of order $\sqrt K$ and most of the time up to
extinction the process will be in this regime.
However, on a sufficiently large time scale (exponential in $K$ but
still much smaller than the extinction time), one can observe large
excursions of order $K$ away from this neighborhood.

The present paper is devoted to the description of  typical
trajectories of the process until extinction, and the derivation  
of statistics and time scales of these large
excursions. 

Since  extinction in finite time occurs almost surely, there is no
nontrivial invariant measure. However there is a particular probability measure
called a quasi-stationary distribution which occurs naturally in the study
of the statistics over times scales which are large but smaller than
the extinction time scale (see for example \cite{CCM1}, \cite{CCM2},
\cite{CCMM}). Through an unusual mixing property, the law of the
process started from a fixed state  involves 
this quasi-stationary distribution (see \eqref{melange}). 

Consider  an open subset $A$ of $\rpd$ containing the fixed point
$\xf$. For an initial condition $\vN(0)$  such that $\vN(0)/K$ belongs
to $A$, 
 one can look at the successive times of exit of the process from $K*A =K A\cap \Z^d$.
We will prove that these successive times normalized by a time scale depending
on  $A$ converge when $K$ tends to infinity towards a Poisson
process. The time scale is related to the large deviation potential
of the process. For a typical trajectory, after an exit time of $K*A$
and a rather small ``wandering'' time, the trajectory comes back to a
neighborhood of $\nf$. 

Consider now another open set $G$ such that $ A\Subset
G\Subset\rpd$. The same result holds for the successive times of exit
from $K*G$. However the time scale associated to $G$ is exponentially
larger (in $K$) than the time scale associated to $A$.
In other words, between the times of visit to the exterior of $K*G$ a
typical trajectory has experienced many more visits to the complement
of $K*A$ without leaving $K*G$ followed by returns to a
neighborhood of $\nf$. On the time scale associated to $G$, we prove 
a law of large numbers for the total time the process spends outside $A$.

This can be furthermore generalized to a funnel of open sets whose
intersection is $\xf$ and union $\rpd$. We get a family of increasing
time scales   corresponding to the times of visits to each of these open
sets. On the corresponding time scale, the visits are all
approximately Poissonian.

This complex behavior is illustrated in Figure \ref{fig1} which
gives the result of the simulation of a trajectory in a system with
one specie up to extinction. Note that the time axis is on a log
scale, hence small times are stretched while large times are
compressed. In a linear scale the picture will be almost completely
black. Note also that the simulation time up to extinction grows
exponentially fast with $K$ prohibiting using larger values of $K$. 


\begin{figure}[!t]
\begin{center}
    \includegraphics[width=10cm]{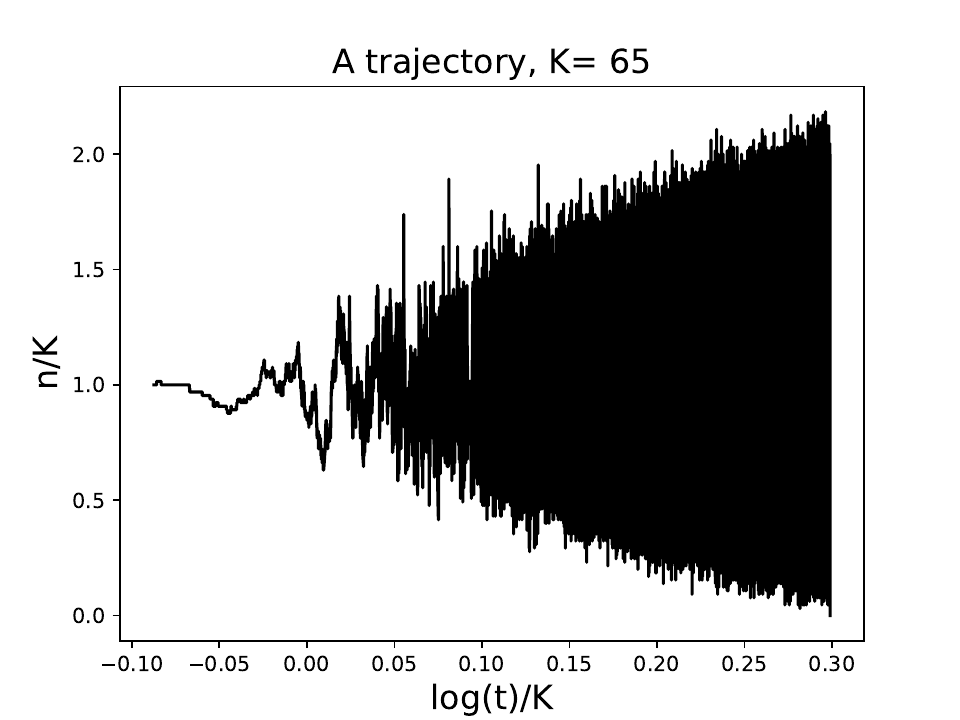}
\end{center}
\caption{A simulation for a one specie model with $K=65$,
$B(x)=2\,x$ , $D(x)=x+x^{2}$ and initial condition $\nf$.}
\label{fig1}
\end{figure}
  
\medskip 
For any $K\ge 1$, we denote by $\tauex$ the extinction time and 
by $\sgK_{t}$ the sub-Markovian  semigroup associated
to $\big(\vN(t)\big)_{t\ge0}$ defined for $\vn\in \domq$ and a measurable bounded function $f$  by 
$$\sgK_{t}f(\vn) = \esperance_{\vn}(f(\vN(t))\, ;\, \tauex>t ).$$

 In  \cite{CCM1},
\cite{CCM2} and \cite{CCMM}, we proved existence and established some
properties of a unique quasi-stationary distribution $\laKqsd$ on $\domq$ (\qsd for short).  (See also \cite{CMSM} and \cite{MV}).
The \qsd provides some information on the  behavior of the process before extinction.

The \qsd 
$\laKqsd$ satisfies  in $\domq$  for any $t\ge0$
\begin{equation}\label{invqsd}
\laKqsd\circ \sgK_{t}=e^{-\lambK\,t}\laKqsd
\end{equation}
and
\begin{equation}\label{expqsd}
\proba_{\laKqsd}\big(\tauex\ge t\big)=e^{-\lambK\,t}\ .
\end{equation}

In particular the \qsd satisfies   
for any subset $G$ of $\domq$
$$
\proba_{\laKqsd}\big(\vN(t)\in G\,\big|\,\tauex\ge t\big)
=\laKqsd\big( G\big)\,.
$$
The extinction rate $\lambK$ 
is  of the form
\begin{equation}
\label{cherlambda}
\lambK=e^{-K\,\Oun}\;.
\end{equation}
The inverse of this number is the expectation of the extinction time in the \qsd which is exponentially large 
 for  large $K$.
There also exists a unique nonnegative right eigenfunction $\lauK$ on $\domq$
 such that for any $t\ge0$
\begin{equation}\label{invu}
\sgK_{t}\,\lauK=e^{-\lambK\,t}\lauK\,.
\end{equation}

 As we have shown in  \cite{CCMM}, many long
time statistics of the process are described using the q.s.d., although the q.s.d. is not  an invariant measure and  the process started in the \qsd is neither stationary nor conservative.
There is however another process, called the q-process. This is  the process $\vN$ conditioned to never being extinct. This q-process is however stationary and mixing, and his invariant probability measure is absolutely continuous with respect to the q.s.d. $\laKqsd$ with density $\lauK$.

\medskip 

 We now describe more precisely the mathematical objects involved in our study.

\medskip The sample space is $\D(\R_+,\R_+^d)$, which is the Skorohod space 
of left-limited and right continuous trajectories on the time set $\R_+$ with values on the state space $\R_+^d$. As
usual it is endowed with the Borel $\sigma-$field associated to the topology of the Skorohod distance, see Chapter 3 in \cite{EK}. We will denote by 
${\cal F}^{K}_{0,\infty}=\sigma({\underline N}^{(K)}(s), s\in [0,\infty))$ the filtration of the process ${\underline N}^{(K)}$.

\medskip 
We will consider  open sets   containing $\xf$ on which we impose the  hypotheses described below, which follow the assumptions 6.1-6.5  in \cite{SW} 
p. 133 .

\noindent 
{\bf Assumption (HS)} : 
\emph{ A subset $A\subset \rpd$ satisfies Assumption ${\bf (HS)}$ if $A$  is an open connected and  bounded set   such that 
 $\xf\in A$ and 
with compact closure  contained in the interior of $\R_{+}^d$, abbreviated by  $A\Subset \R_{+}^d$.  We assume  moreover that $\partial A$ is regular, and  that
$$
\inf_{\vx\in\partial A} \langle \vX(\vx)\,,\, {\vn}^{A}_{\vx}\rangle >0
$$ 
where ${\vn}^{A}_{\vx}$ denotes the inner normal to $\partial A$ at
the point $\vx\in \partial A$.}\\
\emph{This property implies that  the open set  $A$ is invariant by
the flow $\varphi_{\vsbullet}$ of $\vX$.}\\

\medskip For  $A\subset \rpd$, we define
$$
K*A = \{\vn\in \Z^d; \ \frac{\vn}{K} \in A\}
=(K\,A)\cap\Z^{d} \;.
$$

\medskip 
For an open set $A$, we will denote by $\TK_{A}$ the first exit time of the process $\vN/K$
from $A$, defined as the hitting time of the complement  $A^c$,  with $\TK_{A}=0$ if the process starts outside $K*A$, namely 
$$
\TK_{A} = \inf\left\{t\ge 0, {\vN(t)\over K}\notin A\right\}\;.
$$

The successive
times of exit of the process $\vN/K$ from $A$ are well defined but when leaving $A$,
the process $\vN/K$ may come in and out several times in a short time interval
before going far away from the boundary. Describing the details of
this behavior is outside of the scope of this paper. Instead we will
consider a less precise description on a clumping time scale
$\etaK(A)$ (a terminology used for example  in Aldous \cite{Aldous}). 
We will only ask whether or not the process  $\vN/K$ reached the
outside of $A$
during successive intervals of length $\etaK(A)$ .

We have at this point  five time scales which will turn out to be ordered as  below
\begin{center}
\begin{tabular}[]{ll}
jump time scale of the process & $\approx 1/K$\\
$\ll$ unit time scale &  1\\
$\ll$ clumping time scale & $\etaK(A)$\\
$\ll$ rare event time scale of $A$ & $\esperance_{\laKqsd}(\TK_{A}) $ \\
$\ll$ extinction  time scale & $1/\lambK$.\\
\end{tabular}
\end{center}

\medskip
Given a clumping time scale $\etaK(A)>0$,
the clumping structure  of the process is described through a 
sequence of Bernoulli random variables $\big(X_{j}(A,\,K)\big)_{j\in \mathds{Z}_{+}}$ defined
by 
\begin{equation}
\label{def-xjak}
X_{j}(A,\,K)=\begin{cases}1&\;\mathrm{if}\; 
\NK(t)\in K*A^c,\,\mathrm{for\ some}\; t\in [j\etaK(A),\,
  (j+1)\etaK(A)[\;,\\
0&\;\mathrm{otherwise}\;.
\end{cases} 
\end{equation}
Due to almost sure extinction, we have almost surely $X_{j}(A,\,K)=1$
for $j$ large enough.

For $M$ an integer, our main result concerns the $M$-first  clumping time intervals with an excursion outside $A$. This
number for the  time  interval $[0, \, \etaK(A)\, M]$ is equal to 
$$
\sum_{j=0}^{M}X_{j}(A,\,K)\, .
$$

Our main result can be described as follows skipping the  
hypotheses  {\bf (HV)}  stated in details in
Subsection \ref{hypo}.

\begin{theorem}\label{MainBD}
Let $A$    be  a subset of $\rpd$ satisfying Assumption  {\bf (HS)}. 
Under the hypotheses {\bf (HV)}, 
for any $K>1$ 
one can find  a number
$\etaK(A)$ and an integer $\MK(A)$ such that for any fixed $s>0$ 
$$
\lim_{K\to\infty}d_{\mathrm{TV}}\left(
\mathscr{L}^{(\proba_{\laKqsd})}\left(\sum_{j=0}^{[\MK(A)\,s]}X_{j}(A,K)\right)
\,,\,\mathrm{Poisson}(s)\right)=0\;,
$$
where
$
\mathscr{L}^{(\proba_{\laKqsd})}\left(\sum_{j=0}^{[\MK(A)\,s]}X_{j}(A,K)\right)
$
denotes the law of the random variable $\sum_{j=0}^{[\MK(A)\,s]}X_{j}(A,K)$ defined from the process $\vN$ starting from $\laKqsd$, $\mathrm{Poisson}(s)$ denotes a Poisson law of parameter $s$ and $d_{\mathrm{TV}}$ is the
total variation distance.
\end{theorem}

We refer to   Subsection \ref{preuve:BD} for a particular choice
of the two quantities $\etaK(A)$ and  $\MK(A)$.\\

A similar result can also be obtained for the process issued from $[K \vx_{0}]$, with $\vx_{0}\in A$  (see Proposition \ref{part:point}).

The other main results of the paper are 
\begin{enumerate}[1)]
\item   convergence of rescaled $\TK_{A}$  to an exponential law (Theorems \ref{loiexpo} and \ref{loiTAK}), 
\item  asymptotics of the  averaged escape time and its comparison with the clumping time scale in relation with the large deviation potential (Theorems  \ref{echelleTAK} and \ref{equivtout}), 
\item asymptotics of the relative  time spent by the process outside $K*A^c$ on large time scales (Theorem \ref{gdechelle}), 
 \item convergence of the normalized record profiles (Theorem \ref{conv:rec}, see Figure \ref{figtout}-bottom).
\end{enumerate}

As was observed by Antonio Galves the trajectory of the process (see Figure \ref{fig1})  is reminiscent of the phenomenon of
metastability.   The exit to the other phase in metastability would correspond to the
extinction of the population taking place on the largest time scale
$1/\lambK$. Before this catastrophic event  the process makes many
unsuccessful large 
deviations attempts with time scales corresponding to how far it
deviates from $\nf$. The number of these rare events is asymptotically Poissonian
distributed in the above sense. Although there are many analogies with
metastability (see for example \cite{CGOV}, \cite{G2}, \cite{CGE},
\cite{FS1}, \cite{FS2}, \cite{La}, \cite{MS}, \cite{GS}), there are also main differencies since the process is neither stationary nor conservative and a \qsd is not an invariant measure (the unique invariant ergodic probability measure is the Dirac mass at origin), implying a non-standard form of the decay of correlations. For these reasons we cannot apply directly known results on metastability and exponential law for rare events (as for example \cite{IMD}).

The rest of the paper is organized as follows.    In Section \ref{BD}, we recall some properties of the \qsd and draw some consequences on the  decay of correlations. We also choose  $\etaK(A)$ and  $\MK(A)$  based on large deviations results also developed in this section. In Section \ref{q-processus} (Theorems \ref{phimelange} and \ref{thm:mixing}), we introduce the q-process and show that it is $\phi$-mixing (see \cite{bradley} for definition). This  allows (using \cite{LChen}) to obtain the convergence in law (under the distribution of the  stationary q-process) of the number of clumping events on the time intervals $[0, \, \etaK(A)\, \MK(A)s]$  to a Poisson process of parameter $s$. The main result, Theorem \ref{MainBD}, follows by a comparison between the stationary q-process and the process started in the \qsd (see Lemma \ref{connect}, Formula \eqref{distancetv}). In Section \ref{sec:4}, we study the asymptotics of $\TK(A)$ suitably normalized and prove the convergence to an  exponential law.  In Section \ref{sec:5}, we consider the trajectories on the different exponential time scales. This is  summarized by the record 
process and its convergence. 

Our large deviations results are based on the ones of Shwartz and Weiss  \cite{SW}, see also Kratz and Pardoux \cite{KP}. However we had to deal with two difficulties. The first one is that our vector fields are not uniformly Lipschitz continuous neither bounded. The  needed results of Shwartz and Weiss hold nevertheless using  Theorem \ref{lipG}. The second difficulty is Assumption 6.6 in \cite{SW} p.133, that we don't know how to check in our case. However one can check that this assumption is not used  in the proof of Theorem 6.17 (i) p.136 in \cite{SW}
which the only part we use. 
Another difficulty was that one cannot apply directly the large
deviations results of \cite{SW} to the q-process because the rates in
the q-process do not have the adequate functional dependence in $K$.

\subsection{Standing assumptions ${\bf (HV)}$}\label{hypo}

For $\vx\in \rpd$, we use the following standard norms:
\[
\|\vx \|_{1}=\sum_{j=1}^{d} x_{j}\,, \; 
\| \vx \|=\sqrt{\sum_{j=1}^{d} x_{j}^{2}}.
\]

We will assume once for all the following hypotheses.

\begin{description}
\item \label{H1} (HV1) - The vectors fields $\vB$ and $\vD$ vanish only at $\underline 0$. 
\item \label{H2} (HV2) - There exists $\xf$ belonging to the interior of $\rpd$ \textup{(fixed point of $\vX$)} such that
\[
\vB(\xf)-\vD(\xf)=\vX(\xf)= {0}\,.
\]
\item\label{H3} (HV3) -  Attracting fixed point: there exist $\beta>0$ and $R>0$ such that $\| \xf  \|<R$, and for all $\vx\in \rpd$ with $ \| \vx \|<R$,
\begin{equation*}
\langle\, \vX(\vx), (\vx-\xf)\rangle \le -\beta\,  \| \vx \|\,   \|\vx-\xf\, \|^2.
\end{equation*}
\item\label{H4} (HV4) -  The fixed point $\underline 0$ of the vector field $\vX$ is repelling (locally unstable).
Moreover, on the boundary of $\rpd$, the vector field $\vX$ points toward the interior (except at $\underline 0$).
\item\label{H5} (HV5) - Define
\[
\widehat{B}(y)=\sup_{\| \vx\|_{1}=y}\sum_{j=1}^{d}B_{j}(\vx)\,,\;  \widehat{D}(y)=
\inf_{\| \vx\|_{1}=y}\sum_{j=1}^{d} D_{j}(\vx)
\]
and for $y>0$, let
\[
F(y)= \frac{\widehat{B}(y)}{\widehat{D}(y)}.
\]

We assume that there exists $\,0<L<R\,$ such that $\sup_{y>L} F(y)<1/2$ and $\lim_{y\to+\infty} F(y)=0.$
\item\label{H6} (HV6) - There exists $y_{0}>0$ such that $\int_{y_{0}}^{\infty} \widehat{D}(y)^{-1}d y<+\infty$ and $y\mapsto \widehat{D}(y)$ is 
increasing on $\left[y_{0},+\infty\right[$.
\item (HV7) - 
There exists $\xi>0$ such that
\begin{equation*}
\inf_{\vx\,\in \rpd} \inf_{1\leq\, j\,\leq\, d}\, \frac{D_{j}(\vx)}{\sup_{1\leq\, \ell\, \leq \,d} x_{\ell}} >\xi.
\end{equation*}
\item (HV8) - 
Finally, we assume that
\begin{equation*}\inf_{1\leq\, j\,\leq \, d}\partial_{x_{j}} B_{j}(\underline 0) > 0.
\end{equation*}
(By $\partial_{x_{j}}$ we mean the partial derivative with respect to $x_{j}$.)
\end{description}

Assumptions (HV5) and (HV6) ensure that the time for ``coming down from infinity'' for the 
dynamical system is finite. Together with (HV3), this also implies that $\xf$ is a 
globally attracting stable fixed point on $\rpd\backslash\{\vecz\}$. More comments on these assumptions can be found in
\cite{CCM2}. \\

Note that under these assumptions {\bf (HV)}, and if $A$ is a subset satisfying {\bf (HS)}, then the assumptions $6.1-6.5$ of \cite{SW} p.133 are satisfied. 

\bigskip 

\bigskip 

{\bf Aknowledgments : } Servet Martinez was supported by the Center for Mathematical Modeling ANID Basal
Project FB210005. Sylvie M\'el\'eard was supported by the Chair Mod\'elisation Math\'ematique et Biodiversit\'e of Veolia - Ecole
polytechnique - Museum national d'Histoire naturelle - Fondation X. She was   also funded by the European
Union (ERC AdG SINGER, 101054787). Views and opinions expressed are however those of the author(s)
only and do not necessarily reflect those of the European Union or the European Research Council. Neither
the European Union nor the granting authority can be held responsible for them.

\section{Some useful estimates  for the
birth and death processes.}\label{BD}

\subsection{Previous results on the quasi-stationary distribution}

In Theorem 3.3 in \cite{CCM2} 
we proved that for any $t\ge 0$,
\begin{align}\label{melange}
& \sup_{\vn\in\domq}\left|\, \proba_{\vn}(\vN(t)\in \cdot) -
 \e^{-\lambK\, t} \lauK(\vn)\,
\laKqsd(\cdot)-\big(1- \e^{-\lambK t} \lauK(\vn)\big)
\delta_{\vecz}(\cdot)\right|_{TV}\le  \Gamma(t,K)
\end{align}
where
\begin{equation}\label{borneGamma}
\Gamma(t,K)= C_{1}\left( \e^{-C_{2} \;K} \;e^{-\lambK\,t}
+  \e^{-  \frac{C_{3}}{\log K} t}\right) \;,
\end{equation}
the positive constants $C_{1}$, $C_{2}$ and $C_{3}$  being independent of $K$.

\medskip Let us emphasize  that there are additional terms, with respect to the standard stationary case: the additional term $\big(1- \e^{-\lambK t} \lauK(\vn)\big)
\delta_{\vecz}$ in \eqref{melange}  and the first unusual term $ \e^{-C_{2} \;K} \;e^{-\lambK\,t}$ in the definition of  $\Gamma(t,K)$ (in the stationary case,  $
  \e^{-\lambK\, t} =1= \lauK(\vn)$).

\medskip We also have for some constants $C''>0$ and $1>C'''>0$ independent of $K$ (see \cite{CCM1} and \cite{CCM2}) that 
\begin{equation}\label{supuK}
C'''< \inf_{{\vn}} \lauK(\vn)\le \sup_{{\vn}} \lauK(\vn) \le 1+e^{-C'' \,K}\;, \quad \mathrm{and}\quad \int_{\domq}
\lauK\;d\laKqsd=1\;.
\end{equation}

\medskip 

It follows from \cite{CCM1} or \cite{CCM2} that there exists $C>0$
such that for any $K\ge 1$
\begin{equation}
\label{minoration}
 \laKqsd\big(\{\nf\}\big)\ge \frac{C}{K^{d/2}}\;.
\end{equation}

\subsection{Decorrelation estimates.}

\medskip We first give some estimates on time decorrelation which will be useful later. 

In all what follows, we will denote for any $a,b\in \R_{+}\cup{\infty}$ by  $$\mathcal{F}^K_{a,b} = \sigma(\vN(s), a\le s\le b)$$ the filtration of the process between times $a$ and $b$. 
 
\medskip 
Using the Markov property it easily follows that
if  $0<t'<t''$ and  
${A}\in\mathscr{F}_{t',t''}^{K}$, 
 we have (conditioning by $\mathscr{F}_{0,t'}^{K}$)
$$
\esperance_{\laKqsd}\big(\un_{{A}}\;\un_{\tauex>t''}
\big)=e^{-\lambK\,t'}
\esperance_{\laKqsd}\big(\un_{\theta_{-t'}{A}}\;\un_{\tauex>t''-t'}\big)
$$
where $\theta_{\sbullet}$ denotes the translation in time. This
implies the following lemma.

\begin{lemma}\label{decalgen}
For  $0<t'<t''$, and  
${A}\in\mathscr{F}_{t',t''}^{K}$ we have
$$
\left|\esperance_{\laKqsd}\big(\un_{{A}}
\big)-e^{-\lambK\,t'}
\esperance_{\laKqsd}\big(\un_{\theta_{-t'}{A}}\big)\right|
\le
2\,\lambK\,t''\;.
$$
\end{lemma}
\begin{proof}
We have using \eqref{expqsd}
\begin{eqnarray*}
&&\left|\esperance_{\laKqsd}\big(\un_{{A}}
\big)-e^{-\lambK\,t'}
\esperance_{\laKqsd}\big(\un_{\theta_{-t'}{A}}\big)\right|\\
&& \le \left|\esperance_{\laKqsd}\big(\un_{{A}}
\un_{\tauex\le t''}\big) + \esperance_{\laKqsd}\big(\un_{{A}}\un_{\tauex> t''}\big)-e^{-\lambK\,t'}
\esperance_{\laKqsd}\big(\un_{\theta_{-t'}{A}}\big)\right| \\
&&\le  \proba_{\laKqsd}\big(\tauex\le  t''\big) +  \left|e^{-\lambK\,t'}
\esperance_{\laKqsd}\big(\un_{\theta_{-t'}{A}}\;\un_{\tauex>t''-t'}\big)-e^{-\lambK\,t'}
\esperance_{\laKqsd}\big(\un_{\theta_{-t'}{A}}\big)\right| \\
&&\le \proba_{\laKqsd}\big(\tauex\le  t''\big) + e^{-\lambK\,t'}  \proba_{\laKqsd}\big(\tauex\le  t''-t'\big)\\
&&\le  \proba_{\laKqsd}\big(\tauex\le  t''\big) + e^{-\lambK\,t'} e^{\lambK t'}  \proba_{\laKqsd}\big(\tauex\le t''\big) \\
&& \le 2  \proba_{\laKqsd}\big(\tauex\le  t''\big)  = 2(1-e^{-\lambK t''})
\le 2 \lambK t''.
\end{eqnarray*}
\end{proof}

\begin{theorem}\label{RBD}
Under the hypotheses {\bf (HV)},
 we have
{\small \begin{eqnarray*}
&{}& \sup_{\substack{t_{2}>t_{1}+s>t_{1}>0\\A_{1}\in \mathscr{F}^{K}_{0,t_{1}}\,;\,
A_{2}\in\mathscr{F}_{ t_{1}+s,t_{2}}^{K}}}\frac{
\left|\esperance_{\laKqsd}\big(\un_{{A}_{1}}\;\un_{{A}_{2}}\;\un_{\tauex>t_{2}}
\big)-e^{-\lambK\,s}
\esperance_{\laKqsd}\big(\un_{{A}_{1}}\,
\lauK\big(\vN(t_{1})\;\un_{\tauex>t_{1}}\big)\; 
\esperance_{\laKqsd}
\big(\un_{\theta_{-t_{1}-s}{A}_{2}}\;\un_{\tauex>t_{2}-t_{1}-s}\big)\right|}
{\esperance_{\laKqsd}\big(\un_{{A}_{1}}\big)}\\
&{}&
 \hskip 0.8cm \le \Gamma(s,K)\;,
\end{eqnarray*}}
where $ \Gamma(s,K)$ has been defined in \eqref{borneGamma}.
\end{theorem}
\begin{proof}
We have
$$\esperance_{\laKqsd}\big(\un_{{A}_{1}}\;\un_{{A}_{2}}\;\un_{\tauex>t_{2}}
\big) = \esperance_{\laKqsd}\big(\un_{{A}_{1}}\;\un_{\tauex>t_{1}}\;  \esperance_{\vN(t_{1})}\big(\un_{\tauex>s} f(\vN(s)))
\big)
$$
where
$$f(\vN(s)) = \esperance_{\vN(s)}
\big(\un_{\theta_{-t_{1}-s}{A}_{2}}\;\un_{\tauex>t_{2}-t_{1}-s}\big).$$
Note that $f(0)= 0$. Hence 
$$ \esperance_{\vN(t_{1})}\big(\un_{\tauex>s}\; f(\vN(s))
\big) =  \esperance_{\vN(t_{1})}\big( f(\vN(s))
\big).$$
By the estimate \eqref{melange}, we deduce that
$$
\left| \esperance_{\vN(t_{1})}\big( f(\vN(s))
\big) - e^{-\lambK s} \lauK(\vN(t_{1}))\;\laKqsd(f)\right| \leq
\Gamma(s,K)\; .
$$
Then  integrating with respect to 
$\esperance_{\laKqsd}\big(\un_{{A}_{1}}\;\un_{\tauex>t_{1}} \sbullet)$, we obtain

\begin{align*}
&\left|\esperance_{\laKqsd}\big(\un_{{A}_{1}}\;\un_{{A}_{2}}\;\un_{\tauex>t_{2}}
\big) -e^{-\lambK\,s}
\esperance_{\laKqsd}\big(\un_{{A}_{1}}\,
\lauK(\vN(t_{1}))\;\un_{\tauex>t_{1}}\big)\; 
\esperance_{\laKqsd}
\big(\un_{\theta_{-t_{1}-s}{A}_{2}}\;\un_{\tauex>t_{2}-t_{1}-s}\big)\right|\\
&\leq \Gamma(s,K)\;\esperance_{\laKqsd}\big(\un_{{A}_{1}}\;\un_{\tauex>t_{1}}).
\end{align*}
\end{proof}

\subsection{ Large deviation bounds   }

In this section, we are largely inspired by the book of Shwartz and Weiss  \cite{SW} (LDP, Theorem 5.1 and  Kurtz Lemma, Theorem 5.3) and the paper of Kratz-Pardoux  \cite{KP}. Note that in these works, the Large Deviation Principle and the Kurtz Lemma are proved when  the logarithms of the jump rates are assumed to be bounded and globally Lipschitz, which is not our case. Nevertheless, by a coupling argument, we prove   that restricting our process to a compact subset will allow us to apply the above mentioned results in  \cite{SW}.

\medskip 
 Recall that $(\varphi_{s})_{s}$ is the semiflow associated with $\vX$. 
\begin{theorem}
\label{lipG}
Let $A$ and $G$ be open subsets of $\rpd$ satisfying  Assumption {\bf (HS)}  and such that $A \Subset G$.   Then there exists a positive constant $K(A,G)$ such that for any $K\ge K(A,G)$,

(i) there exists a positive constant $ C(A,G)$ with
\begin{align*}
\sup_{{\vn ;\,  d(\frac{\vn}{K},A)\le  {1\over 2}  d(\partial A, \partial G)} }\proba_{\vn} \big(\TK_{G}\le 1 \big) &\le e^{-C(A,G)\,K}\,; \\
\sup_{\vn \in K*A }\proba_{\vn}\big( \vN(1)\in K*A^{c}\big)& \le e^{-C(A,G)\,K} \,;
\end{align*}
(ii) there exist positive constants  $C_{1}(A,G),   C_{2}(A,G),  \zeta(A,G)$ such that  for any $0<\zeta\le \zeta(A,G)$,
$$
\sup_{\vz;\, d(\vz,A)\le {1\over 2}d(\partial A, \partial G) }\proba_{[\vz K]} \Big(\sup_{0\le t \le 1} \left\| \frac{\vN(t)}{K}-\varphi_{t}(\vz)\right\| >\zeta\Big) \le C_{1}(A,G)\,e^{-C_{2}(A,G) \,K\,\zeta}.
$$
\end{theorem}

\begin{proof}
In order to estimate 
$$
\proba_{\vn} \big(\TK_{G}\le 1 \big)
$$
for $\vn$ such that $d(\frac{\vn}{K},A)\le  {1\over 2}  d(\partial A, \partial G)$, 
 we have to first replace the rates $\vB$ and $\vD$ by
bounded smooth ones. 

Le $\psi_{G}$ be a smooth function with values in $[0,1]$ such that
$$
\psi_{G}(\vx)=\begin{cases}
1& \mathrm{if} \;\vx\in G\;,\\
0& \mathrm{if}\; d(\vx,\bar G) \ge  {1\over 2}d(\bar G, \,\partial \rpd)\;.\\
\end{cases}
$$

We now define new birth and death rates by
$\vB^G(\vx)=\psi_{G}(\vx)\,\vB(\vx)+(1-\psi_{G}(\vx)) \vun$ and
$\vD^G(\vx)=\psi_{G}(\vx)\,\vD(\vx)+(1-\psi_{G}(\vx)) \vun$. These rates are bounded and
uniformly Lipschitz. We denote by $\vM$ the associated birth
and death process. We now construct a coupling between $\vN$ and $\vM$
as follows. Recall that (see for example \cite{KP})
$$
\vN(t)=\vN(0)+
\sum_{j=1}^{d}P_{j}\left(\int_{0}^{t}K\,B_{j}\big(\vN(s)/K\big)\;ds\right)
-\sum_{j=1}^{d}P_{j+d}\left(\int_{0}^{t}K\,D_{j}\big(\vN(s)/K\big)\;ds\right)\;,
$$ 
where $(P_{j}(\vsbullet))_{{1\le j\le 2d}}$ are $2d$
 independent standard Poisson processes. The
coupling between $\vN$  and $\vM$ is realized by using the same
Poisson processes, namely
$$
\vM(t)=\vM(0)+
\sum_{j=1}^{d}P_{j}\left(\int_{0}^{t}K\, \vB^G_{j}\big(\vM(s)/K\big)
\;ds\right)
-\sum_{j=1}^{d}P_{j+d}\left(\int_{0}^{t}K\,\vD^G_{j}\big(\vM(s)/K\big)
\;ds\right)\;.
$$
If $\vN(0)=\vM(0)$, the two coupled processes coincide at least until
they exit $K*G$. We have for $\vn$ such that $d(\frac{\vn}{K},A)\le  {1\over 2}  d(\partial A, \partial G)$,
$$
\proba_{\vn}\big(T^K_{G} \le 1\big)
=\proba_{\vn}\big(\exists \,s\in(0,\,1]\,,\;
\vM(s)\notin  K*G\big)\;.
$$
Let
$$
\varepsilon_{A,G}=
\min\big\{  {1\over 2}  d(\partial A, \partial G),\;d(\overline{\varphi_{1}(A)}, A^{c})\big\}\;.
$$
It is easy to verify that $\varepsilon_{A,G}>0$. 
We then have from the definition of $\varepsilon_{A,G}$ and for any $\vn$ such that $d(\frac{\vn}{K},A)\le  {1\over 2}  d(\partial A, \partial G)$, that
$$
\proba_{\vn}\big(\exists \,s\in(0,\,1]\,,\;
\vM(s)\in K*G^c\big)\le \proba_{\vn}\Big(\sup_{s\in[0,1]}
\Big\|\frac{1}{K} \vM(s)-\varphi_{s}\Big(\frac{\vn}{K}\Big)\Big\|\ge \varepsilon_{A,G}\Big)\;.
$$
Applying Kurtz's Lemma  (see \cite{SW} Theorem 5.3) to the process $\vM/K$  we deduce
that there exits a constant $C'(A,G)>0$ independent of $K$ such that
for  $K$ large enough and any $\vn $  such that $d(\frac{\vn}{K},A)\le  {1\over 2}  d(\partial A, \partial G)$, 
$$
\proba_{\vn}\Big(\sup_{s\in[0,1]}
\Big\|\frac{1}{K} \vM(s)-\varphi_{s}\Big(\frac{\vn}{K}\Big)\Big\|\ge
 \varepsilon_{A,G}\Big)\le e^{-K\,C'(A,G)}\;.
$$ 
This implies that
for  $K$ large enough,
\begin{equation}
\label{eq:tau}
\sup_{\vn;\,d(\frac{\vn}{K},A)\le  {1\over 2}  d(\partial A, \partial G)}\proba_{\vn}\big(\TK_{G}\le 1\big) \le e^{-K\,C'_{W}}\;.
\end{equation}
Furthermore, for $\vn \in K*A$,
\begin{align*}
\proba_{\vn}\big( \vN(1)\in K*A^{c}\big) 
&\le \proba_{\vn}\big(\TK_{G}\le 1\big) + \proba_{\vn}\big( \vN(1)\in K*A^{c} ; \TK_{G}> 1\big)\\
&\le \proba_{\vn}\big(\TK_{G}\le 1\big) + \proba_{\vn}\big( \vM(1)\in K*A^{c}\big).
\end{align*}
Point (i) follows since for $\vn \in K*A$, 
$$ \proba_{\vn}\big( \vM(1)\in K*A^{c}\big) \le \proba_{\vn}\Big(\Big|\frac{1}{K} \vM(1)-\varphi_{1}\Big(\frac{\vn}{K}\Big)\Big|\ge
 \varepsilon_{A,G}\Big)\le e^{-K\,C'(A,G)}.$$

Let us now prove (ii) in a similar way. For $\vz$ such that $d(\vz,A)\le {1\over 2}\,d(\partial A, \partial G)$ , we have

\begin{align*}
&\proba_{[\vz K]} \left(\sup_{0\le t \le 1} \left
\| \frac{\vN(t)}{K}-\varphi_{t}(\vz)\right\| >\zeta\right)\\
&\le \proba_{[\vz K]}\big(\TK_{G}\le 1\big)  
+ \proba_{[\vz K]} \left(\sup_{0\le t \le 1}\left\| 
\frac{\vN(t)}{K}-\varphi_{t}(\vz)\right\| >\zeta\,;\, \TK_{G} >1\right)\\
&\le \proba_{[\vz K]}\big(\TK_{G}\le 1\big)  
+ \proba_{[\vz K]} \left(\sup_{0\le t \le 1} \left\| 
\frac{\vM(t)}{K}-\varphi_{t}(\vz)\right\| >\zeta\,;\, \TK_{G} >1\right).
\end{align*}
The result follows as above.
\end{proof}

\begin{corollary}\label{iter:lipG}
Under the hypothesis of Theorem \ref{lipG}, for any $K>K(A,G)$ and any
integer $p$
$$
\sup_{\vn\in K*A}\proba_{\vn}\big(\TK_{G}\le p\big) \le p\;e^{-C(A,G)\,K} \;.
$$
\end{corollary}
\begin{proof}

The case $p=1$ follows immediately from Theorem \ref{lipG}. We now
assume $p>1$. For $\vn\in K*A$, we have by the Markov property
\begin{align*}
\proba_{\vn}\big(\TK_{G}>p\big)=\esperance_{\vn}\big(\un_{\TK_{G}>1}
\esperance_{\vN(1)}\big(\un_{\TK_{G}>p-1}\big)\big)
&\ge
\esperance_{\vn}\big(\un_{\TK_{G}>1}\big)\;
\inf_{\vn\in K*A}\esperance_{\vn}\big(\un_{\TK_{G}}>p-1\big)
\\&
\ge \big(1-\proba_{\vn}\big(\TK_{G}\le 1\big)\big)
\,\inf_{\vn\in K*A}\proba_{\vn}\big(\TK_{G}> p-1\big)\big)\;.
\end{align*}
We conclude recursively that for any $p\ge1$
$$
1\ge \inf_{\vn\in K*A}\proba_{\vn}\big(\TK_{G}>
p\big)\ge \left(1-e^{-C(A,G)\,K}\right)^{p}\;.
$$
This implies
$$
\sup_{\vn\in K*A}\proba_{\vn}\big(\TK_{G}\le p\big) \le
1-\left(1-e^{-C(A,G)\,K}\right)^{p} 
$$
and the result follows. 
\end{proof}

\begin{corollary}\label{iter:Kurtz}
Let $G$ be an open subset of $\rpd$ satisfying Assumption 
{\bf (HS)}. Let $\vx\in G$. Then there exists 
constants $K(\vx,G)>1$ and $C(\vx,G)>0$ such that for any integer $p$
and any $K>K(\vx,G)$
$$
\proba_{[K\,\vx]}\big(\TK_{G}\le p\big) \le p\;e^{-C(\vx,G)\,K}\;.
$$
\end{corollary}
\begin{proof}
Using the normal bundle to the (regular) boundary of $G$, it is easy
to construct an open set $A\Subset G$, satisfying  hypothesis 
{\bf (HS)} with $\vx\in A$. The corollary follows from Corollary  \ref{iter:lipG}.
\end{proof}

\begin{corollary}\label{largedevW}
Let $G\subset \rpd$ be an open subset 
 satisfying Assumption {\bf (HS)}. Then there
exist two constants $C_{G}>0$ and $K_{G}>0$ such that for any $K>K_{G}$
$$
{\laKqsd}\big( K*G^c\big)\le e^{-K\,C_{G}}\;.
$$
\end{corollary}

\begin{proof}
Due to the regularity of $\partial \bar G$, one can show that there exists an open subset $A$ satisfying Assumption {\bf (HS)} and  $A\Subset G$.

For any $\vn\in K*A$ we have
\begin{align*}
&\proba_{\vn}\big(\vN([K^{2}])\in K*G^c\big)\le
\proba_{\vn}\big(\vN([K^{2}])\in K*A^{c}\big)
\\&
\le
\proba_{\vn}\big(\vN([K^{2}]-1)\in
K*A^{c}\big)+\proba_{\vn}\big(\vN([K^{2}]-1)\in K*A,\, 
\vN([K^{2}])\in K*A^{c}\big)
\end{align*}
and iteratively
$$
\proba_{\vn}\big(\vN([K^{2}])\in K*G^c\big)\le
\sum_{j=0}^{[K^{2}]-1}\proba_{\vn}\big(\vN(j)\in K*A,\, 
\vN(j+1)\in K*A^{c}\big)\;.
$$
By the Markov property we have
$$
\proba_{\vn}\big(\vN(j)\in K*A,\, 
\vN(j+1)\in K*A^{c}\big)
$$
$$
=\esperance_{\vn}\big(\un_{K*A}(\vN(j))\;
\proba_{\vN(j)}\big(\vN(1)\in K*A^{c}\big)\big)\;.
$$ 
Then
\begin{equation}
\label{eq:ouf}
\proba_{\vn}\big(\vN([K^{2}])\in K*G^c\big)
\le K^2 \sup_{\vm\in K*A}\proba_{\vm}\big(\vN(1)\in K*A^{c}\big)\,.
\end{equation}

We now apply the estimate  \eqref{melange} with $t=[K^{2}]$ and
$\vn=\nf$.

 We
get, since $\laKqsd(\vecz)=0$,
$$
\left|\, \proba_{\nf}\big(\vN([K^{2}])\in K*G^c\backslash \{\vecz\}\big) -
 \e^{-\lambK\, [K^{2}]} \lauK(\nf)\,
\laKqsd(K*G^c) \right|
\le C_{1}\left( \e^{-C_{2} \;K} 
+  \e^{- \frac{C_{3}}{\log K} [K^{2}]}\right)\;.
$$
This implies,   using \eqref{eq:ouf} and Theorem \ref{lipG} that 
$$
\e^{-\lambK\, [K^{2}]} \lauK(\nf)\, \laKqsd(K*G^c)
\le 2\, K^2\,e^{-K\,C'_{G}}+C_{1}\left( \e^{-C_{2} \;K} 
+  \e^{- \frac{C_{3}}{\log K}   [K^{2}]}\right)\;.
$$
The result follows from Theorem 3.2  and Proposition 7.5 in 
\cite{CCM2} recalled in \eqref{cherlambda} and \eqref{supuK}.

\end{proof}

\begin{corollary}\label{Kurtz}
For any open subset $G$ of $\rpd$ satisfying hypothesis {\bf (HS)},
 there exist 
constants  $D_{1}(G)>0$ and $D_{2}(G)>0$, $K'_{G}>1$ 
such that for any $K>K'_{G}$
$$
\proba_{\laKqsd}\big(\TK_{G}<1)\le D_{1}(G)\,e^{-K\,D_{2}(G)}\;.
$$
\end{corollary}
\begin{proof}
We introduce as in the previous proof an open subset $A$ satisfying Assumption {\bf(HS)} and $A\Subset G$. 
We have
 \begin{align*}
\proba_{\laKqsd}\big(\TK_{G}<1) 
&\le  \proba_{\laKqsd}\big(\TK_{G}<1 ; \vN(0) \in K*A\big) 
+ \proba_{\laKqsd}\big( \vN(0) \notin K*A\big) \\
&\le \sum_{\vn\in K*A}  \laKqsd(\vn)  \proba_{\vn}\big(\TK_{G}<1\big)
 + \laKqsd(K*A^c\big)\\
&\leq \sup_{\vn\in K*A} \proba_{\vn}\big(\TK_{G}<1\big) + \laKqsd(K*A^c\big)\\
&\le e^{-K\,C'_{A^c}} + e^{-K\,C_{A^c}},
\end{align*}
for $K$ large enough, where we have used \eqref{eq:tau} and Corollary
 \ref{largedevW}.

\end{proof}

\medskip 
Consider  an integer $\delta_{K}\ge 1$ which will be chosen later on and an open 
subset $A$ of $\rpd$ with $\xf\in A$.
We introduce the random variable 
$Z_{0}(A,K)$ defined by
\begin{equation}
\label{def:Z0}
Z_{0}(A,\,K)=\begin{cases}1&\;\mathrm{if}\; 
\NK_{t} \in K*A^c,\,\mathrm{for\ some}\; t\in [0,\,
  \delta_{K}[\;,\\
0&\;\mathrm{otherwise}\;.
\end{cases} 
\end{equation}
\begin{lemma}\label{borne:Z0}
Assume $\limsup_{K\to\infty}\log(\delta_{K})/K=0$. Then 
for any open subset $A$ of $\rpd$ satisfying Assumption {\bf (HS)}, 
 there exist $K_{A}>1$ and $\beta_{A}>0$ such that for any $K>K_{A}$  
$$
\proba_{\laKqsd}\big( Z_{0}(A,K)=1\big)\le e^{-\beta_{A}\,K}\;.
$$
\end{lemma}
\begin{proof}
We define   a 
 sequence of Bernoulli random variables $\big(\vartheta_{j}(A,\,K)\big)_{j\in \Z_{+}}$ 
by 
$$
\vartheta_{j}(A,\,K)=\begin{cases}1&\;\mathrm{if}\; 
\NK_{t}\in K*A^c,\,\mathrm{for\ some}\; t\in [j,\,(j+1)[\;,\\
0&\;\mathrm{otherwise}\;.
\end{cases} 
$$
By Bonferroni's inequality we have from the definition of $
Z_{0}(\AK,K)$  that
$$
\proba_{\laKqsd}\big( Z_{0}(A,K)=1\big)\le \sum_{j=0}^{\deltaK}
\proba_{\laKqsd}\big(\vartheta_{j}(A,\,K)=1\big)\;.
$$
By Lemma \ref{decalgen} and Corollary \ref{Kurtz}  we get
$$
\proba_{\laKqsd}\big( Z_{0}(A,K)=1\big)\le D_{1}(A)\,\deltaK
e^{-K\,D_{2}(A)}+2\,\lambK\,\deltaK^2\;.
$$
The result follows from \eqref{cherlambda}. 
\end{proof}

\subsection{ Large deviation potential and Extinction time   }

For an open set  $B\varsubsetneq\rpd$ satisfying 
 hypothesis {\bf (HS)},  let 
$V_{*}(B)$ denote the large deviation
potential for $B$ with respect to $\xf$, namely
$$
V_{*}(B)=\inf_{y\in\partial B}V(\xf,\,y)\;,$$
where as usual (see for example \cite{KP} Section7)
$$V(x,\,y) =\inf_{t>0} \inf\{I_{0}^t(\phi): \phi\in \D([0,t],\R_+^d), \phi(0)=x, \phi(t)=y\},$$
with $I_{0}^t$  the large deviation functional and $\D([0,t],\R_+^d)$ is the
Skorohod space with trajectories restricted to the time set $[0,t]$.

We refer to \cite{SW} p.136  and \cite{KP} Section 7 for the definitions and properties
of $V$ and $V_{*}$.

By convention, we also define $V_{*}(\rpd)$ as
$$V_{*}(\rpd)= \sup_{B \, {\rm satisfying}\  {\bf(HS)}} V_{*}(B).$$

The following result is intuitively obvious but we couldn't find the proof of this statement in the literature. 

\begin{theorem}\label{ordreV}
Let $A$ and  $G$ be  open sets satisfying Hypothesis {\bf (HS)} and 
$ A\Subset G$. Then
$$
0<V_{*}(A)<V_{*}(G)\;.
$$
\end{theorem}

The  proof is postponed to  \ref{appendice1}.\\

The large deviation potential allows to give a bound for the extinction rate. 

 \begin{lemma}\label{lambdaV}
Under Hypotheses {\bf{(HV)}},
$$
\limsup_{K\to\infty}\frac{\log \lambK}{K}\le -V_{*}(\rpd)\;.
$$

\end{lemma}

\begin{proof}
Recalling \eqref{minoration}, i.e.
$$
\proba_{\laKqsd}\big(\{\nf\})\ge \frac{C}{K^{d/2}}
$$
for some $C>0$ independent of $K$, we have 
$$
\frac{1}{\lambK}=\esperance_{\laKqsd}\big( \tauex\big)\ge
\frac{C}{K^{d/2}}\; \esperance_{\nf}\big( \tauex\big)
\ge \frac{C}{K^{d/2}}\; \esperance_{\nf}\big( \TK_{B}\big)
\;,
$$
for any  subset $B$ of $\rpd$ satisfying {\bf (HS)}, since obviously $ \TK_{B} \leq \tauex$.
The result follows from the large deviation estimates of the exit
time, see for example 
\cite{SW} Theorem 6.17 (ii). Indeed, for any such  $B$ of  $\rpd$, we have  that
$$\limsup_{K\to \infty} \frac{1}{K}\log \esperance_{\nf}\big( \TK_{B}\big) = V_{*}(B).$$
The result follows.
\end{proof}

\begin{remark}
 Using  Remark \ref{1dformule} of  \ref{appendice1} and the precise estimate  for $\lambK$ established in
\cite{CCM1} for $d=1$, we get
$$
\lim_{K\to\infty}\frac{\log \lambK}{K}= -V_{*}(\R_{+})\;.
$$
\end{remark}

\subsection{Estimation of $\laKqsd\big(|1-\lauK\big|\big)$}

\begin{lemma}\label{petit}
Let $\nu$ be a probability measure and $u$ a nonnegative function
such that for some constant $\sigma>0$, $u\le 1+\sigma$. Assume also
$$
\int u\;d\nu=1\;.
$$
Then 
$$
\int \big|1-u\big|\;d\nu\le 2\,\sigma\;.
$$
\end{lemma}

\begin{proof}
We have
\begin{align*}
\int \big|1-u\big|\;d\nu&=\int_{u\le 1} d\nu-\int_{u\le 1} u\; d\nu
+\int_{u>1} u\; d\nu-\int_{u>1}d\nu
\\&
=1-\int_{u>1}d\nu-1+\int_{u>1} u\; d\nu+\int_{u>1} u\;
d\nu-\int_{u>1}d\nu
\\&
=2\,\int_{u>1} (u-1)\; d\nu\le2\, \sigma\,\nu\big(\{u>1\}\big)\;,
\end{align*}
and the result follows.
\end{proof}

\begin{corollary}
\label{starstar} For the constant $C''$ given in \eqref{supuK} and $K$ large enough, we have
\begin{equation*}
\int_{\domq}\laKqsd\big(|1-\lauK\big|\big) \le 2 e^{-C''K}.
\end{equation*}

\end{corollary}

This is a direct consequence of Lemma \ref{petit} and \eqref{supuK}.

\subsection{Choice of $\etaK(A)$, $\MK(A)$, $\deltaK$
 and consequences}\label{preuve:BD}

Let $A$ a subset of $\rpd$ satisfying Assumption ${\bf (HS)}$.

 \medskip We need to choose $\etaK(A)$, $\MK(A)$  and $\deltaK$
 adequately.

 This will involve  estimates of $\Gamma(\deltaK, K)$ (see
 Theorem \ref{RBD}), $\proba_{\laKqsd}\big( Z_{0}(A,K)=1\big)$ and 
$\laKqsd\big(|1-\lauK\big|\big)$.
The strategy is to choose first $\deltaK$, 
choose $\MK(A)$, then $\etaK(A)$,
 estimate $\proba_{\laKqsd}\big(X_{0}(A,K)=1\big)$ and verify that 
$\deltaK<\etaK(A)$. At this point, the random variables $X_{j}(A,K)$ and $Z_{0}(A,K)$ will be completely defined. \\

First of all we need to choose $\deltaK$ satisfying the assumption of
Lemma \ref{borne:Z0}, namely
 $$\limsup_{K\to\infty}\log(\delta_{K})/K=0\;.
$$
To fix ideas, we choose once for all 
\begin{equation}\label{choix:deltaK}
\deltaK=\big[K^{2}\big]\;.
\end{equation}
Other choices are possible.

\medskip
Choose  $\alpha_{A}$ such that
\begin{equation}
\label{alpha}
0<\alpha_{A}<\gamma_{A}=\frac{1}{\log 2}\;
\min\left\{ \beta_{A},\,C'', \,\frac{-\log\lambK}{3\,K},\, C_{2},
 \,\frac{V_{*}(\rpd)-V_{*}(A)}{2}\right\}
\end{equation}
with  $\beta_{A}$ defined in Lemma \ref{borne:Z0}, $C''$  defined in 
 \eqref{supuK}  and $C_{2}$ in  \eqref{borneGamma}.

The fact that 
$V_{*}(\rpd)>V_{*}(A)$, and hence $\gamma_{A}>0$ is a consequence of
 Theorem \ref{ordreV}. 

\medskip
We define $\MK(A)$ by
\begin{equation}\label{choix:MKA}
\MK(A)=2^{[\alpha_{A}\,\sK]}\,.
\end{equation}
Note that by Lemma \ref{borne:Z0}, $\MK(A)\, 
\proba_{\laKqsd}\big(Z_{0}(A,K)=1\big)$ tends to
zero exponentially fast when $K$ tends to infinity. We also have
\begin{equation}\label{borne:MKA}
\MK(A)< \min\left\{e^{\beta_{A}\,\sK},\;
  e^{C''\,\sK},\;\lambK^{-1/3}, e^{C_{2}K}, \;e^{\,\sK\,(V_{*}(\rpd)-V_{*}(A))/2}
\right\} \;e^{-\zeta_{A}\,\sK}\;, 
\end{equation}
for some $\zeta_{A}=\gamma_{A} - \alpha_{A} >0$ if $K$ is large enough.

\medskip
We now choose $\etaK(A)$  as
$$
\etaK(A)=\JK\,\deltaK
$$
where $\JK$ is  the largest integer such that
\begin{equation}\label{choix:etaK}
\proba_{ \laKqsd}\big(\TK_{A}\le \JK\,\deltaK\big)=
\proba_{\laKqsd}\big(X_{0}(A,K)=1\big)\le \frac{1}{\MK(A)}\;.
\end{equation}

\begin{lemma}\label{encadre}
We have
$$
\etaK(A)\le e^{K\,(V_{*}(A)+o(1))}\;,
$$
and
\begin{align}
\label{inspiration}
\frac{1}{\MK(A)}\;
\big[1-\MK(A)\proba_{\laKqsd}\big(Z_{0}(A,K)=1\big)-2\,\MK(A) 
\lambK (\etaK(A)+\deltaK)\big]\; &\le \proba_{\laKqsd}\big(X_{0}(A,K)=1\big)\nonumber \\&\le \frac{1}{\MK(A)}\;.
\end{align}
\end{lemma}

\medskip
\begin{proof}
From our choice \eqref{choix:MKA} of $\MK(A)$ we have
$$
\proba_{ \laKqsd}\big(\TK_{A}\le \etaK(A)\big)=
\proba_{\laKqsd}\big(X_{0}(A,K)=1\big)
\le \frac{1}{\MK(A)}\le 2^{-\alpha_{A}\,K+1}\;.
$$
We have also 
$$
\proba_{ \laKqsd}\big(\TK_{A}\le \etaK(A)\big)\ge
  \laKqsd\big(\{\nf\}\big)\;\proba_{ \nf}\big(\TK_{A}\le \etaK(A)\big) \;.
$$

Therefore, using \eqref{minoration}, 
$$
\proba_{ \nf}\big(\TK_{A}\le \etaK(A)\big)\le C^{-1} \,K^{d/2} \;2^{-\alpha_{A}\,K+1}
$$
which implies from \cite{SW} Theorem 6.17 (i)
$$
\etaK(A)\le e^{K\,(V_{*}(A)+o(1))}\;.
$$

\medskip
From the definition of $\JK$ it follows that
$$
\proba_{ \laKqsd}\big(\TK_{A}\le \JK\,\deltaK\big)\le \frac{1}{\MK(A)}
\le \proba_{ \laKqsd}\big(\TK_{A}\le (\JK+1)\,\deltaK\big)\;.
$$
We have
\begin{align*}
&\proba_{\laKqsd}\big(\TK_{A}\le (\JK+1)\,\deltaK\big)
\\&
=\proba_{ \laKqsd}\big(\TK_{A}\le \JK\,\deltaK\big)+
\proba_{\laKqsd}\big(\JK\,\deltaK\le \TK_{A}\le (\JK+1)\,\deltaK\big)\;.
\end{align*}
We have using Lemma \ref{decalgen}
\begin{align*}
&\proba_{\laKqsd}\big(\JK\,\deltaK\le \TK_{A}\le (\JK+1)\,\deltaK\big)\\
&\le \proba_{\laKqsd}\big(\big\{\exists \, t\,; \,\JK\,\deltaK< t\le
(\JK+1)\,\deltaK \, and\, \, \NK_{t}/K\notin A\big\}\big)\\
&\le \proba_{\laKqsd}\big(Z_{0}(A,K)=1\big)+2\, \lambK (\JK+1)\,\deltaK\;.
\end{align*}
\
The lower bound for $\proba_{\laKqsd}\big(X_{0}(A,K)=1\big)$  follows immediately and the second statement holds.
\end{proof}

\begin{remark}
 Note that  by Lemma \ref{lambdaV} and the choices made in \eqref{alpha} and \eqref{choix:MKA}, we have 
 for any $A\subset \rpd$ satisfying Assumption ${\bf (HS)}$ that
\begin{equation}
\label{lambdaeta}
\lim_{K\to \infty}e^{ \lambK\, \etaK(A)} =1
\end{equation}
and 
 \begin{equation}
 \label{asympto}
 \lambK\,  \MK(A) \, \etaK(A) \le e^{-K\Big((V_{*}(\rpd)-V_{*}(A))/2 +o(1)\Big)}.
 \end{equation}
\end{remark}

\medskip
We now deduce from Lemma \ref{lambdaV} that the time interval $\etaK(A)$ involved in the definition of the $X_{j}(A,K)$  is strictly larger than the time interval $\deltaK$ involved in the definition of $Z_{0}(A,K)$.

\begin{lemma}\label{deltaeta}
For any $K$ large enough, we have 
$$
\deltaK<\etaK(A)\;.
$$
\end{lemma}
\begin{proof}
If $\etaK(A)\le \deltaK$ we must have
$$
\proba_{ \laKqsd}\big(\TK_{A}\le \etaK(A)\big)<
\proba_{ \laKqsd}\big(\TK_{A}\le
\deltaK\big)=\proba_{\laKqsd}\big(Z_{0}(A,K)=1\big)\;.
$$
Using Lemma \ref{borne:Z0} and the second statement of Lemma
\ref{encadre}, we get a contradiction for large $K$ with our choice 
\eqref{choix:MKA} of $\MK(A)$.
\end{proof}

\section{Proof of Theorem \ref{MainBD}
 and of other results.} \label{sec:3}

\subsection{The q-process and some of its 
 properties.}\label{q-processus}

We
recall that the q-process is the time-homogeneous Markov process on the set
$\domq$ whose transition probability is given
for $t>0$ by
\begin{equation}\label{transq}
\QK\big(\vN(t)=\vn\big|\vN(0)=\vm\big)=e^{\lambK\,t}
\frac{\lauK(\vn)}{\lauK(\vm)}\,\proba_{\vm}\big(\vN(t)=\vn\big)\;,
\end{equation}
see for example \cite{CMSM} Section 5.5 and all references therein. This process is never extinct, it is stationary,
ergodic
with invariant probability measure $\mu_{\sK}$, which is  given by
\begin{equation}\label{lamuK}
\frac{d\mu_{\sK}}{d\laKqsd}=\lauK.
\end{equation}

\begin{lemma}
\label{rateq} The q-process is  a Birth and Death process, with   a birth rate for  the $j$th specie ($1\le j\le d$)
 at $\vn\in\domq$ given by 
$$
K\,\frac{\lauK(\vn+\ve_{j})}{\lauK(\vn)} \,B_{j}(\vn/K)
$$
and a death rate
$$
\begin{cases}
&K\,\frac{\lauK(\vn-\ve_{j})}{\lauK(\vn)} \,D_{j}(\vn/K)\; \ 
\mathrm{if} \ \;\vn-\ve_{j}\in\domq\;,\\
&0\;\mathrm{else}\;.
\end{cases}
$$
\end{lemma}

\begin{proof} It is obvious that the q-process is a birth and death process. Its birth and death rates are immediately obtained 
 by derivating \eqref{transq}
 at $t=0$.
\end{proof}

The absorbed process is concentrated on the set of trajectories

$$\Omega_0=\{\vs\in \D(\R_+,\R_+^d):\vs(t)=\vecz \Rightarrow \vs(u)=\vecz\,, \forall \, u \ge t\},$$ namely
for all ${\vn}\in \Z_+^d\backslash \{\vecz\}$, we have $\proba_{\vn}(\Omega_0)=1$. We also define 
$$\Omega^*= \{\vs\in \D(\R_+,\R_+^d): \vs(t)\neq\vecz ,\,\forall \, t>0\}$$

\begin{remark}
\label{R12}
For all ${\vn}\in \Z^{d}_{+}\backslash\{\vecz\}$, $\proba_{\vn}(\Omega^*)=0$ since the process  $\vN$ goes a.s. to extinction, 
 and $\QK_{\vn}(\Omega^*)=1$, since the $q-$process is never extinct. 
\end{remark}

For any $t\ge 0$ we will denote by ${\cal F}^{K,*}_{0,t}$ the restriction of the field ${\cal F}^K_{0,t}$ to the set
$\Omega^*_t= \{\vs\in \D(\R_+,\R_+^d): \vs(u)\neq\vecz \hbox{ for } u\in [0,t]\}$.

\medskip 
The jump rates of the q-process given in Lemma \ref{rateq} do not  have  the functional form assumed 
in \cite{SW} or \cite{KP}.
Therefore we cannot apply
directly the results of these works. 
It is however possible to connect  results on the process starting in
the \qsd 
and results on
the q-process due to the following lemma.

\begin{lemma}\label{connect}
For any $t\ge0$ and $\vs\in \Omega_{*}$,
\begin{equation}
\label{radon}
\frac{d\QK}{d\proba}\Big|_{\mathscr{F}_{0,t}^{K,*}}(\vs)
=\frac{e^{\lambK\,t}\lauK(\vs(t))}{\lauK(\vs(0))}.
\end{equation}

For any $K>0$ and $t\ge0$ 
and  ${\cal A}\in\mathscr{F}_{0,t}^{K,*}$ ,
\begin{equation}
\label{distancepq}
\big|\QK_{\mu_{\sK}}\big({\cal A}\big)
-\proba_{\laKqsd}\big({\cal A}\big)\big|\leq  \int_{\domq} |1-\lauK|\,d\laKqsd+\left(e^{\lambK\,t}-1\right).
\end{equation}
Let $t(\sbullet)$ be a positive function on $\R_{+}$ such that
\begin{equation}
\label{lambdak}
\lim_{K\to\infty}\lambK\,t(K)=0\;.
\end{equation}
Then
\begin{equation}
\label{distancetv}
\lim_{K\to\infty}d_{\mathrm{TV}}\left(\proba_{\laKqsd}\big|_{\mathscr{F}_{0,t(K)}^{K,*}}\,,\,
\QK_{\mu_{\sK}}\big|_{\mathscr{F}_{0,t(K)}^{K,*}}\right)=0\;.
\end{equation}

\medskip 
Moreover, there exists $D>1$ independent of $K$ such that 
if $(F_{\sK})_{\sK\in\R_{+}}$ is a family of positive functions
 on  $\Omega_{*}$ such that  for each
$K>0$, $F_{K}$ is measurable with respect to $\mathscr{F}_{0,t(K)}^{K,*}$, then
for any $K>0$
\begin{equation}
\label{comparison}
\frac{1}{D}\,\int_{\Omega_{*}} F_{K} \,d \proba_{\laKqsd} \le 
\int_{\Omega_{*}} F_{K} \,d \QK_{\mu_{\sK}} \le D\,\int_{\Omega_{*}}
F_{K} \,
d \proba_{\laKqsd}\;.
\end{equation}
\end{lemma} 

\begin{proof}
Equation \eqref{radon} immediately follows from \eqref{lamuK}.
For any  $K>0$ and any  
$\ 
{\cal A}\in\mathscr{F}_{0,t(K)}^{K,*}\;,
 $
we have by \eqref{transq} and  \eqref{lamuK}
\begin{align*}
\QK_{\mu_{\sK}}\big({\cal A}\big)&=e^{\lambK\,t(K)}
\,\esperance_{\laKqsd}\big(\un_{{\cal A}}\;\lauK(\vN(t(K))\big)\\
&=\esperance_{\laKqsd}\big(\un_{{\cal A}}\big)
+e^{\lambK\,t(K)}\big(\esperance_{\laKqsd}\big(\un_{{\cal A}}\;
\lauK(\vN(t(K))\big)-\esperance_{\laKqsd}\big(\un_{{\cal
    A}}\big)\big)\\
&\hskip 0.5cm +\left(e^{\lambK\,t(K)}-1\right)\,\esperance_{\laKqsd}\big(\un_{{\cal
    A}}\big)\big)\;.
\end{align*}
This implies
\begin{align*}
\big|\QK_{\mu_{\sK}}\big({\cal A}\big)
-\proba_{\laKqsd}\big({\cal A}\big)\big|&\le
 e^{\lambK\,t(K)}\esperance_{\laKqsd}\big(\big|1-\lauK(\vN(t(K))\big|\big)
+\left(e^{\lambK\,t(K)}-1\right)
\\&
\le \int_{\domq} |1-\lauK|\,d\laKqsd+\left(e^{\lambK\,t(K)}-1\right)\;,
\end{align*}
by using \eqref{invqsd}.
Therefore
\begin{align*}
d_{\mathrm{TV}}\left(\proba_{\laKqsd}\big|_{\mathscr{F}_{0,t(K)}^{K,*}}\,,\,
\QK_{\mu_{\sK}}\big|_{\mathscr{F}_{0,t(K)}^{K,*}}\right)
&=\sup_{{\cal
  A}\in\mathscr{F}_{0,t(K)}^{K,*}}
\big|\QK_{\mu_{\sK}}\big({\cal A}\big)
-\proba_{\laKqsd}\big({\cal A}\big)\big|
\\&\le \int_{\domq} |1-\lauK|\,d\laKqsd+\left(e^{\lambK\,t(K)}-1\right)
\end{align*}
which tends to zero when $K$ tends to infinity by the hypothesis \eqref{lambdak} and 
Corollary \ref{starstar}. This proves the first part of the
lemma.\\

For the second part we have
$$
\int_{\Omega_{*}} F_{K} \,d \QK_{\mu_{\sK}} =e^{\lambK\,t(K)}\;
\esperance_{\laKqsd}\Big(F_{K}\big(\vN(s), s\le t(K)\big)\,\lauK(\vN(t(K))\big)\Big)
$$
and the result follows from
$$
1\le \inf_{\sK}e^{\lambK\,t(K)}\le \sup_{\sK}e^{\lambK\,t(K)}< +\infty
$$
and \eqref{supuK}.
\end{proof}

We now show that the q-process is exponentially 
$\phi$-mixing.

\begin{theorem}\label{phimelange}
There exist three   constants $K_{*}>1$, $C_{1}'>0$  and $C_{3}'>0$ 
 such that for any $K\ge K_{*}$, for any $\vm$ in $\domq$ and any $t\ge0$
\begin{align*}
\sup_{\vm\in\domq}d_{TV}\big(\QK_{\vm}\big(\vN(t)\in\vsbullet \big)
\,,\,\mu_{\sK}(\vsbullet)\big) &= \sup_{\vm\in\domq}\sum_{\vn\in\domq} \left|\QK_{\vm}\big(\vN(t)\in\vn \big)
-\mu_{\sK}(\vn)
\right|\\
&\le C_{1}'\;e^{-C_{3}'\,t/\log K}\;.
\end{align*}
\end{theorem}

For the proof of Theorem \ref{phimelange} we need  preparatory
Lemmas.

\begin{lemma}\label{mieuxcor}
(i) For any integer $p\ge 1$ and any $\zeta>0$ and any $K\ge 1$,
\begin{align*}
&\sup_{\vm\in\domq}\sum_{\vn\in\domq}\left|\proba_{\vm}\big(\vN(p\,
 \zeta \log K)=\vn\big)
-e^{-p\,\lambK\, \zeta \log K} \; u_{\sK}(\vm)\;\laKqsd(\vn)
\right|\\
&= \sup_{\vm\in\domq}\sum_{\vn\in\domq}\left|\sgK_{p \zeta \log
  K}(\vm,\vn)
-e^{-p\,\lambK\, \zeta \log K} \; u_{\sK}(\vm)\;\laKqsd(\vn)
\right|\\
&
\le \Gamma( \zeta \log K,\,K)^{p}\;.
\end{align*}

(ii) It follows that for any $K$ large enough
$$
\sup_{\vm\in\domq}\sum_{\vn\in\domq}\left|\proba_{\vm}\big(\vN(
\deltaK)=\vn\big)
-e^{-\lambK\,\deltaK} \; u_{\sK}(\vm)\;\laKqsd(\vn)
\right|\le \lambK\;.
$$
\end{lemma}

\begin{proof}

(i) Let $\mathcal{R}_{*}$  be the rank one operator in
$\ell^{\infty}(\domq)$ with kernel 
$$
\mathcal{R}_{*}(\vm,\,\vn)=e^{-\lambK\, \zeta \log K} \; u_{\sK}(\vm)\;\laKqsd(\vn)\;.
$$
For a real bi-infinite matrix $T$ with indices in $\domq$, we define

$$
\normiii{ T}  = \sup_{\vm\in\domq}\sum_{\vn\in\domq}\left|
T(\vm,\vn)\right|.
$$

Note that for two bi-infinite matrices  $T_{1}$ and $T_{2}$, 
$$
\normiii{T_{1}T_{2}}  \le \normiii{ T_{1}}
\, \normiii{ T_{2}}.
$$
We have from  \eqref{melange} 
$$
\normiii[\Big]{  \sgK_{ \zeta \log
  K}-\mathcal{R}_{*}}
 \le \Gamma( \zeta \log K,\,K)\;.
$$
From \eqref{invqsd}, \eqref{cherlambda} and \eqref{invu} we have
$$
\mathcal{R}_{*}\;\sgK_{ \zeta \log K}=\sgK_{ \zeta \log K}\mathcal{R}_{*}=
e^{-\lambK\, \zeta \log K} \mathcal{R}_{*}\;.
$$
We have also from \eqref{supuK}
$$
\mathcal{R}_{*}^{2}=e^{-\lambK\, \zeta \log K} \mathcal{R}_{*}\;.
$$
It follows using
$$
\mathcal{R}_{*}\;\big(\sgK_{ \zeta \log K}-\mathcal{R}_{*}\big)=
\big(\sgK_{ \zeta \log K}-\mathcal{R}_{*}\big)\; \mathcal{R}_{*}=0
$$
that for any $p\ge 1$
$$
\normiii[\Big]{\big(\sgK_{ \zeta \log K}\big)^{p}-e^{-\lambK\, \zeta \log K\,(p-1)}
\mathcal{R}_{*}}
 \le \Gamma( \zeta \log K,\,K)^{p}\;,
$$
which proves the first part of 
the lemma.

(ii) We check from \eqref{borneGamma} that $\,\Gamma( [\zeta \log K]+1,K) \le
C_{1}(e^{-C_{2} K}e^{-\lambK\, \zeta \log K} + e^{-C_{3} \zeta })$.  We choose $\zeta^*$ independent of $K$  such that for $K$ large enough
$\, \Gamma( [\zeta^* \log K]+1,K) \le 1/ e$. Then there exists an integer $p$ such that  $p( [\zeta^* \log K]+1) = \deltaK$ . The results follows.
\end{proof}

\begin{lemma}\label{mieuxmelange}
There exist two  constants $C_{1}''>0$ and $C_{3}''>0$ 
 such that for any $t\ge0$
\begin{equation}
\label{eq:mieux}
\sup_{\vm \in \domq}d_{TV}\Big(\proba_{\vm}\big(\vN(t)\in\vsbullet\big),
e^{-\lambK\,t}\lauK(\vm)\,\laKqsd(\vsbullet)\Big)
\le C_{1}''\;e^{-C_{3}''\,t/\log K}\;.
\end{equation}
\end{lemma}
\begin{proof}
As in the previous proof, we
choose $\zeta^*$ independent of $K$  such that for $K$ large enough
$\, \Gamma( \zeta^* \log K,K) \le 1/ e$. It follows from Lemma
\ref{mieuxcor} that  \eqref{eq:mieux} is proved with $C_{1}'' = 1$ and
$C_{3}'' = 1 /\zeta^*$ for any $t$ of the form $p\,\zeta^* \log K$. 

To obtain  \eqref{eq:mieux}  for a  general $t>\zeta^* \log K$, we write
$ \, t= p\,\zeta^* \log K + s$, with $0\le s<\zeta^* \log K$ for a
positive $p$. We have by the semigroup property
\begin{align*}
&\un_{\ell^{\infty}(\domq)}\,\sgK_{t}
-e^{-\lambK\,t}\,\lauK\otimes\laKqsd
\\&=
\un_{\ell^{\infty}(\domq)}\sgK_{s}
\left(\un_{\ell^{\infty}(\domq)}\sgK_{p\,\zeta^* \log K}-
e^{-\lambK\,p\,\zeta^* \log K}\,\lauK\otimes\laKqsd\right)
\\& \hskip 1cm
+e^{-\lambK\,p\,\zeta^* \log
  K}\left(\un_{\ell^{\infty}(\domq)}\sgK_{s}\lauK\otimes\laKqsd
-e^{-\lambK\,s}\,\lauK\otimes\laKqsd\right).
\end{align*}
We observe that the second term vanishes by \eqref{invu}. Therefore
using Lemma \ref{mieuxcor} with $\zeta=\zeta^{*}$ we get
\begin{align*}
&\normiii{ \,\sgK_{t}
-e^{-\lambK\,t}\,\lauK\otimes\laKqsd}
\\
&\le \sup_{0\le s\le \zeta^* \log K}
\normiii{\,\sgK_{s}}\, 
e^{-p}\le e^{-p\,\zeta^* \log K/(\zeta^* \log K)}
\\
&\le 
e^{-t/(\zeta^* \log K)}\sup_{0\le s\le \zeta^* \log K}e^{s/(\zeta^*
  \log K)}\le e\;e^{-t/(\zeta^* \log K)}\;,
\end{align*}
which is the estimate \eqref{eq:mieux} with $C_{1}''=e$ and $C_{3}'' =
1 /\zeta^*$. 

Finally, to obtain  \eqref{eq:mieux}  for  $0\le t\le \zeta^* \log K$,
we observe that
\begin{align*}
\normiii{ \,\sgK_{t}
-e^{-\lambK\,t}\,\lauK\otimes\laKqsd}
&\le
\normiii{ \,\sgK_{t}}
+\normiii{ 
e^{-\lambK\,t}\,\lauK\otimes\laKqsd} \le 2+e^{-C''\,K}\le 3 
\end{align*}
by \eqref{supuK}. We can write
\begin{align*}
\normiii{ \,\sgK_{t}
-e^{-\lambK\,t}\,\lauK\otimes\laKqsd} &\le 3 
e^{-t/(\zeta^* \log K)}\;e^{t/(\zeta^* \log K)}
\le 3\,e\, e^{-t/(\zeta^* \log K)}
\end{align*}
which is the estimate \eqref{eq:mieux} with $C_{1}''=3\, e$ and
 $C_{3}'' = 1 /\zeta^*$. We have obtained  \eqref{eq:mieux} for
 all $t\ge 0$ with $C_{1}''=3\, e$ and
 $C_{3}'' = 1 /\zeta^*$.
\end{proof}

\begin{proof}[Proof of Theorem \ref{phimelange}.]
It follows immediately from Lemma
\ref{mieuxmelange}, the bounds \eqref{supuK} and formulas
\eqref{transq} and \eqref{lamuK} that for any $\vm\in\domq$
$$
d_{TV}\big(\QK_{\vm}\big(\vN(t)\in\vsbullet\big)
,\mu_{\sK}(\vsbullet)\big)\le
e^{\lambK\,t}\;\frac{2}{C'''}\;C_{1}''\;e^{-C_{3}''\,t/\log K}.
$$
From \eqref{cherlambda} there exists $K_{*}>1$ such that for any $K\ge
K_{*}$, $\lambK\le C_{3}''/(2\log K)$.
Theorem  \ref{phimelange} follows  with
$C_{1}'=2\,C_{1}''/C''' $ and $C_{3}'=C_{3}''/2 $.
\end{proof}

We deduce the exponential $\phi$-mixing property.  

\begin{theorem}
\label{thm:mixing}
Let us consider $0<t_{1}<t_{2}$ and $B\in \mathcal{F}^{K,*}_{t_{2},\infty} $. Then we have
$$\left | \QK_{\mu_{\sK}}\Big(B \,|\, \mathcal{F}^K_{0,t_{1}}\Big)-  \QK_{\mu_{\sK}}(B)\right |\leq C'_{1}\,e^{-(t_{2}-t_{1})C'_{3}/\log K}.$$
\end{theorem}

\begin{proof}
\begin{align*}
  \QK_{\mu_{\sK}}\Big(B \,|\, \mathcal{F}^K_{0,t_{1}}\Big) 
 & =\sum_{\vn\in \domq, \vm\in \domq} \QK_{\mu_{\sK}}\Big(B;  \vN(t_{2})=\vn ; \vN(t_{1}) =\vm \,|\, \mathcal{F}^K_{0,t_{1}}\Big)\\
 &= \sum_{\vn\in \domq, \vm\in \domq} \QK_{\mu_{\sK}}\Big(B \,|\, \vN(t_{2})=\vn ;  \vN(t_{1}) =\vm ; \mathcal{F}^K_{0,t_{1}}\Big)\\
& \hskip 2cm \times \QK_{\mu_{\sK}}\Big( \vN(t_{2})=\vn \,|\,\vN(t_{1}) =\vm ; \mathcal{F}^K_{0,t_{1}}\Big)\, \QK_{\mu_{\sK}}\Big(\vN(t_{1}) =\vm  \,|\, \mathcal{F}^K_{0,t_{1}}\Big) \\
&= \sum_{\vn\in \domq, \vm\in \domq} \QK_{\mu_{\sK}}\Big(B \,|\, \vN(t_{2})=\vn \Big)\\
& \hskip 2cm \times \QK_{\mu_{\sK}}\Big( \vN(t_{2})=\vn \,|\,\vN(t_{1}) =\vm \Big)\, \QK_{\mu_{\sK}}\Big(\vN(t_{1}) =\vm  \,|\, \mathcal{F}^K_{0,t_{1}}\Big)\,.
 \end{align*}
Then we  can write
\begin{align*}
 \left | \QK_{\mu_{\sK}}\Big(B \,|\, \mathcal{F}^K_{0,t_{1}}\Big)-  \QK_{\mu_{\sK}}(B)\right |
&\leq
\sum_{\vn\in \domq, \vm\in \domq} \QK_{\mu_{\sK}}\Big(B \,|\, \vN(t_{2})=\vn \Big)\  \QK_{\mu_{\sK}}\Big(\vN(t_{1}) =\vm  \,|\, \mathcal{F}^K_{0,t_{1}}\Big)\\
& \hskip 2cm \times \left| \QK_{\mu_{\sK}}\Big( \vN(t_{2})=\vn \,|\,\vN(t_{1}) =\vm \Big)  - \mu_{K}(\vn)\right |  \\
&\leq  C'_{1}\,e^{-(t_{2}-t_{1})C'_{3}/\log K},
 \end{align*}
where the last inequality results from Theorem \ref{phimelange}.
\end{proof}

\subsection{Convergence to the Poisson process for the $q$-process}

An immediate consequence of Theorem \ref{thm:mixing} is that the sequence $(X_{j}(A,K), j \ge 0)$ satisfies the mixing condition required in Chen \cite{LChen}.

\begin{corollary} 
\label{mix} The  sequence $(X_{j}(A,K), j \ge 0)$ satisfies the following mixing condition. For all $j, k\ge 2$ and any $B\in \sigma(X_{j+\ell}(A,K), \ell \ge k)$, one has
$$\left|\, \QK_{\mu_{\sK}}\Big(B\,|\, \sigma(X_{0}(A,K),\ldots, X_{j}(A,K))\Big) - \QK_{\mu_{\sK}}(B)\right| \leq \phi_{K}(k),$$
where 
$\, \phi_{K}(k) = C'_{1} e^{-C'_{3} (k-1) \etaK(A)/\log K}$. 
\end{corollary}

The proof is  deduced from Theorem \ref{thm:mixing} with $t_{1} = j\,\etaK(A) $ and $t_{2}= (j+k)\etaK(A)$.

\begin{theorem}\label{MainBDq}
Let $A$    be  a subset of $\rpd$ satisfying Assumption  {\bf (HS)}. 
Under the hypotheses {\bf (HV)}, 
for any $K>1$ 
one can find  a number
$\etaK(A)$ and an integer $\MK(A)$  (see Section \ref{preuve:BD}) 
such that for any fixed $s>0$ 
$$
\lim_{K\to\infty}d_{\mathrm{TV}}\left(
\mathscr{L}^{(\QK_{\mu_{\sK}})}\bigg(\sum_{j=0}^{[\MK(A)\,s]}X_{j}(A,K)\bigg)
\,,\,\mathrm{Poisson}(s)\right)=0\;.
$$
\end{theorem}
\begin{proof}
Let
$$
W(K,\,A,\,s)=\sum_{j=0}^{[\MK(A)\,s]}X_{j}(A,K)\;.
$$
We have by the stationarity of the q-process 
$$
\widehat {\mathds E}^{(K)}_{\mu_{\sK}}\big(W(K,\,A,\,s)\big)=s_{\sK}(A)
$$
with 
$$
s_{\sK}(A)=[\MK(A)\,s]\; \widehat {\mathds E}^{(K)}_{\mu_{\sK}}\big(X_{0}(A,K)\big) =[\MK(A)\,s]\; \QK_{\mu_{\sK}}\big(X_{0}(A,K)=1\big) \;.
$$
From Lemma \ref{connect}  \eqref{distancepq}, we have
$$
\big|\QK_{\mu_{\sK}}\big(X_{0}(A,K)=1\big)
-\proba_{\laKqsd}\big(X_{0}(A,K)=1\big)\big| 
\leq \int_{\domq} |1-\lauK|\,d\laKqsd
+\left(e^{\lambK\,\etaK(A)}-1\right)\;.
$$ 

Using $\MK(A)$ and $\etaK(A)$ defined in  \eqref{alpha}, \eqref{choix:MKA} and 
\eqref{choix:etaK} respectively,  Corollary \ref{starstar},  \eqref{supuK} and \eqref{lambdaeta}, we deduce that
$$
\lim_{K\to\infty}\big|[\MK(A)\,s]\;\QK_{\mu_{\sK}}\big(X_{0}(A,K)=1\big)
-[\MK(A)\,s]\;\proba_{\laKqsd}\big(X_{0}(A,K)=1\big)\big|=0\;.
$$

It follows from  Lemma \ref{encadre} using Lemma \ref{borne:Z0} and our
choices of $\MK(A)$, $\deltaK$ and $\etaK(A)$ that  $\, \MK(A)\;\proba_{\laKqsd}\big(X_{0}(A,K)=1\big)$ tends to $1$ as $K$ tends to infinity. We deduce that
\begin{equation}
\label{eq:sKA}
\lim_{K\to\infty} s_{\sK}(A)=s\;,
\end{equation}
and from estimate (2.2) in \cite{AJ} we obtain
$$
\lim_{K\to\infty}d_{\mathrm{TV}}\left(\mathrm{Poisson}(s_{\sK}(A))
\,,\,\mathrm{Poisson}(s)\right)=0\;.
$$

From the triangle inequality for the total variation norm, it will be
enough to prove that
$$
\lim_{K\to\infty}d_{\mathrm{TV}}\left(
\mathscr{L}^{(\QK_{\mu_{\sK}})}\bigg(\sum_{j=0}^{[\MK(A)\,s]}X_{j}(A,K)\bigg)
\,,\,\mathrm{Poisson}(s_{\sK}(A))\right)=0\;.
$$
An estimate on the l.h.s. follows from \cite{LChen} Theorem
4.1. The decorrelation assumption required in \cite{LChen} is proved in  Corollary \ref{mix}.  With our notations  we have (using the
stationarity of the q-process) for the integer  $m=2 $ in \cite{LChen} that 
\begin{align}
\label{eq:formula}
&d_{\mathrm{TV}}\left(
\mathscr{L}^{(\QK_{\mu_{\sK}})}\bigg(\sum_{j=0}^{[\MK(A)\,s]}X_{j}(A,K)\bigg)
\,,\,\mathrm{Poisson}(s_{\sK}(A))\right)\nonumber\\
&\le 6\,\min\big\{s_{\sK}(A)^{-1/2}, 1\big\}\,
 \bigg[\mathrm{Var}(W(K,\,A,\,s))-s_{\sK}(A)+ 10\,\,[\MK(A)\,s]\,\,
\QK_{\mu_{\sK}}\big(X_{0}(A,K)=1\big)^{2}\nonumber \\
&
\hskip 3cm +4\, (s_{\sK}(A)+1)\,[\MK(A)\,s]\; C_{1}'e^{-C_{3}'\,\etaK(A)/\log(K)}
\bigg]\;.
\end{align}

We  show that each term in brackets tends to $0$ as $K$ tends to infinity. 

We have
\begin{align*}
\mathrm{Var}(W(K,\,A,\,s))&=\widehat {\mathds E}^{(K)}_{\mu_{\sK}}\big(W(K,\,A,\,s)^{2}\big)
-\widehat {\mathds E}^{(K)}_{\mu_{\sK}}\big(W(K,\,A,\,s)\big)^{2}
\\&
=\sum_{j=0}^{[\MK(A)\,s]}\widehat {\mathds E}^{(K)}_{\mu_{\sK}}\big(X_{j}(A,K)^{2}\big)
+\sum_{j,\,k=0\,,j\neq
  k}^{[\MK(A)\,s]}\widehat {\mathds E}^{(K)}_{\mu_{\sK}}\big(X_{j}(A,K)\;X_{k}(A,K)\big)
-s_{\sK}(A)^{2}
\\&
=s_{\sK}(A)-s_{\sK}(A)^{2}+2\,\sum_{j,\,k=0\,,j<k
  }^{[\MK(A)\,s]}\widehat {\mathds E}^{(K)}_{\mu_{\sK}} \big(X_{j}(A,K)\,,X_{k}(A,K)\big)\;.
\end{align*}
We have
$$
\sum_{j,\,k=0\,,j<k
  }^{[\MK(A)\,s]}\widehat {\mathds E}^{(K)}_{\mu_{\sK}}\big(X_{j}(A,K)\;X_{k}(A,K)\big)=I_{1}
+I_{2}
$$
with
\begin{align*}
I_{1}&=\sum_{1\le j+1<k\le [\MK(A)\,s]}
\widehat {\mathds E}^{(K)}_{\mu_{\sK}}\big(X_{j}(A,K)\;X_{k}(A,K)\big) =\sum_{1\le j+1<k\le [\MK(A)\,s]}\; \widehat {\mathds E}^{(K)}_{\mu_{\sK}}
\big(X_{0}(A,K)\;X_{k-j}(A,K)\big)\;,
\end{align*}
and by stationarity,
$$
I_{2}=\sum_{j=0}^{[\MK(A)\,s]-1}\widehat {\mathds E}^{(K)}_{\mu_{\sK}}
\big(X_{j}(A,K)\,,X_{j+1}(A,K)\big)=[\MK(A)\,s]\;\widehat {\mathds E}^{(K)}_{\mu_{\sK}}
\big(X_{0}(A,K)\;X_{1}(A,K)\big)\;.
$$
For any $\ell\ge 2$ and conditioning on $\sigma(X_{0})$, we obtain by using Corollary \ref{mix} that 
\begin{align*}
& \left| \QK_{\mu_{\sK}}\big(X_{0}(A,K)=1\ ;\ X_{\ell}(A,K)=1\big) - \QK_{\mu_{\sK}}\big(X_{0}(A,K)=1\big) \QK_{\mu_{\sK}}\big(X_{\ell}(A,K)=1\big) \right|\\
&\le \phi_{K}(\ell-1)\,  \QK_{\mu_{\sK}}\big(X_{0}(A,K)=1\big).
\end{align*}
Therefore
\begin{align*}
&  \bigg| I_{1} -  \sum_{j\le 0\,,1<k-j\le [\MK(A)\,s]-j} \QK_{\mu_{\sK}}\big(X_{0}(A,K)=1\big) \QK_{\mu_{\sK}}\big(X_{k-j}(A,K)=1\big) \bigg|\\
& \le  \QK_{\mu_{\sK}}\big(X_{0}(A,K)=1\big)\, \sum_{j\ge 0\,,1<k-j\le [\MK(A)\,s]-j} \phi_{K}(k-j-1)\,.
\end{align*}
Note that
\begin{align*}
 & \sum_{j\le 0\,,1<k-j\le [\MK(A)\,s]-j} \QK_{\mu_{\sK}}\big(X_{0}(A,K)=1\big)\, \QK_{\mu_{\sK}}\big(X_{k-j}(A,K)=1\big)\\ 
&= \Big(\frac{1}{2}[\MK(A)\,s]^2 + \mathcal{O}([\MK(A)\,s])\Big) 
 \QK_{\mu_{\sK}}\big(X_{0}(A,K)=1\big)^2\\
 &=\frac{1}{2}\, s_{\sK}(A)^2 + s_{\sK}(A) \QK_{\mu_{\sK}}\big(X_{0}(A,K)=1\big)\, \mathcal{O}(1)\\
 &=\frac{1}{2}\, s_{\sK}(A)^2 + \frac{s_{\sK}(A)^2}{ [\MK(A)\,s]} \, \mathcal{O}(1)\,,
\end{align*}
and the last term tends to $0$ by  \eqref{eq:sKA} and \eqref{choix:MKA}.

Let us now remark that 
\begin{align*}
 \QK_{\mu_{\sK}}\big(X_{0}(A,K)=1\big)\,  \sum_{j\ge 0\,,1<k-j\le [\MK(A)\,s]-j} \phi_{K}(k-j-1)&\leq \QK_{\mu_{\sK}}\big(X_{0}(A,K)=1\big)\, \sum_{j=0}^{[\MK(A)\,s]} \sum_{\ell=2}^\infty \phi_{K}(\ell) \\
&=  s_{\sK}(A) C'_{1} \sum_{\ell=1}^\infty e^{-C'_{3}\,\ell \, \etaK(A)/\log(K)}.
\end{align*}
Using Lemma \ref{deltaeta}, we conclude that this quantity tends to $0$ as $K$ tends to infinity. (Recall that $\delta_{K}=[K^2]$). 

Summarizing  the previous results, we get that $\, I_{1} - \frac{1}{2}\, s_{\sK}(A)^2\,$ tends to $0$ as $K$ tends to infinity.

\medskip
In order to estimate $I_{2}$ we introduce the  Bernoulli random variable  $X'_{1}(A,K)$ which is equal to $1$ if $\vN(t)\in A$ for some $t\in [\etaK(A)+\delta_{K}, 2\etaK(A)[$ and $0$ otherwise. 
We observe that $$
\widehat {\mathds E}^{(K)}_{\mu_{\sK}}\big(X_{0}(A,K)\;X_{1}(A,K)\big) = \QK_{\mu_{\sK}}\big(X_{0}(A,K) =1\; ;\,X_{1}(A,K)=1\big) ,$$
and  that
\begin{align*}&\left|  \QK_{\mu_{\sK}}\big(X_{0}(A,K) =1\; ;\,X_{1}(A,K)=1\big) -  \QK_{\mu_{\sK}}\big(X_{0}(A,K) =1\; ;\,X'_{1}(A,K)=1\big)\right| \le 2\,  \QK_{\mu_{\sK}}\big(Z_{1}(A,K) =1\big).\end{align*}
We also have
$$
 \QK_{\mu_{\sK}}\big(X_{1}(A,K)=1\big) = \QK_{\mu_{\sK}}\big(Z_{1}(A,K) =1\big) + \QK_{\mu_{\sK}}\big(Z_{1}(A,K) =0\; ;\,X'_{1}(A,K)=1\big)$$
which implies
\begin{align*} \QK_{\mu_{\sK}}\big(X'_{1}(A,K)=1\big) =&  \QK_{\mu_{\sK}}\big(X_{1}(A,K)=1\big) - \QK_{\mu_{\sK}}\big(Z_{1}(A,K) =1\big)\\
& +  \QK_{\mu_{\sK}}\big(Z_{1}(A,K) =1\; ;\,X'_{1}(A,K)=1\big).\end{align*}
Using stationarity and Theorem \ref{thm:mixing}, we get
\begin{align*} 
&\left|  \QK_{\mu_{\sK}}\big(X_{0}(A,K) =1\; ;\,X_{1}(A,K)=1\big) -  \QK_{\mu_{\sK}}\big(X_{0}(A,K) =1\big)^2\right|\\
& \leq 4\,  \QK_{\mu_{\sK}}\big(Z_{0}(A,K) =1\big) + C_{1}'e^{-C_{3}'\delta_{K})/\log(K)}.
\end{align*}
Note that using \eqref{comparison},
\begin{align*}  \QK_{\mu_{\sK}}\big(X_{0}(A,K) =1\big) &\le D \,\proba_{\laKqsd}\big(X_{0}(A,K) =1\, ;\, \tauex > \etaK(A) \big)\\&  \le D \,\proba_{\laKqsd}\big(X_{0}(A,K) =1 \big)\,.\end{align*}
Then \eqref{eq:sKA}, \eqref{choix:etaK} and \eqref{choix:MKA} imply that $\, [\MK(A)\,s]\, \QK_{\mu_{\sK}}\big(X_{0}(A,K) =1\big)^2\,$ tends to $0$ as $K$ tends to infinity. 
By using in addition \eqref{borne:MKA} and Lemma \ref{borne:Z0} we conclude that $I_{2}$ tends to $0$ as $K$ tends to infinity. 
At this stage we have obtained that $\,\mathrm{Var}(W(K,\,A,\,s))-s_{\sK}(A)\,$ tends to $0$ as $K$ tends to infinity.  
The other terms in the r.h.s. term in \eqref{eq:formula} are estimated in the same way.
This concludes the proof of Theorem \ref{MainBDq}.
\end{proof}

\subsection{Back and forth with  the killed process.}

\subsubsection{From the q-process to the killed process, proof of
  Theorem \ref{MainBD}.}

Let $A$    be  a subset of $\rpd$ satisfying Assumption {\bf (HS)}. 
 We consider $[\MK(A)\,s]$ and $\etaK(A)$ as in  Theorem \ref{MainBDq}.

The following corollary of Theorem \ref{MainBDq} is a proof of Theorem \ref{MainBD}.

\begin{corollary}
Under the hypotheses {\bf (HV)}, 
for any $K>1$,
and  for any fixed $s>0$ 
$$
\lim_{K\to\infty}d_{\mathrm{TV}}\left(
\mathscr{L}^{(\proba_{\laKqsd})}\bigg(\sum_{j=0}^{[\MK(A)\,s]}X_{j}(A,K)\bigg)
\,,\,\mathrm{Poisson}(s)\right)=0\;.
$$
\end{corollary}

\begin{proof}
\begin{align*}
&d_{\mathrm{TV}}\left(
\mathscr{L}^{(\proba_{\laKqsd})}\bigg(\sum_{j=0}^{[\MK(A)\,s]}X_{j}(A,K)\bigg)
\,,\,\mathrm{Poisson}(s)\right)\\
& \leq \sum_{k=0}^\infty \bigg|\proba_{\laKqsd}
\bigg(\sum_{j=0}^{[\MK(A)\,s]}X_{j}(A,K)=k\,;\, 
\tauex> ([\MK(A)\,s]+1)\etaK(A)\bigg) - e^{-s}\,\frac{s^k}{k!}\bigg|\\
& \hskip 0.5 cm+ \proba_{\laKqsd}\big(\tauex\le ([\MK(A)\,s]+1)\etaK(A)\big)\\
& = \sum_{k=0}^\infty
 \bigg|\proba_{\laKqsd}\big|_{\mathscr{F}_{0,([\MK(A)\,s]+1)\etaK(A)}^{K,*}}
\bigg(\sum_{j=0}^{[\MK(A)\,s]}X_{j}(A,K)=k\bigg) - e^{-s}\,
\frac{s^k}{k!}\bigg| \\ 
&\hskip 0.5 cm + \proba_{\laKqsd}\big(\tauex\le ([\MK(A)\,s]+1)\etaK(A)\big)\,.
\end{align*}
For the second term we have from \eqref{expqsd}
$$
 \proba_{\laKqsd}\big(\tauex\le ([\MK(A)\,s]+1)\etaK(A)\big)\le 
\lambK\,([\MK(A)\,s]+1)\etaK(A)
$$
which tends to zero when $K$ tends to infinity from our choices of 
$\MK(A)$ and $\etaK(A)$ (see \eqref{borne:MKA}, Lemma
\eqref{encadre} and Lemma \eqref{lambdaV}).

The result follows using the first part of Lemma \ref{connect} with
$t(K)=([\MK(A)\,s]+1)\etaK(A)$ and Theorem \ref{MainBDq}.

\end{proof}

\begin{proposition}
\label{part:point}
 The same result
follows for the process starting in $[K\vx_{0}]$ (instead of $\laKqsd$)
with $\vx_{0}\in A$ with $A$ satisfying Assumption {\bf (HS)} : for any $s>0$, 
$$
\lim_{K\to\infty}d_{\mathrm{TV}}\left(
\mathscr{L}^{(\proba_{[K\vx_{0}]})}\bigg(\sum_{j=0}^{[\MK(A)\,s]}X_{j}(A,K)\bigg)
\,,\,\mathrm{Poisson}(s)\right)=0\;.
$$
\end{proposition}

\begin{proof}

For any integer $q\ge0$ we have using Corollary \ref{iter:Kurtz} with
$p=\deltaK$ 
\begin{align*}
&\proba_{[K\,\vx_{0}]}\bigg(\sum_{j=0}^{[\MK(A)\,s]}X_{j}(A,K)=q\bigg)
=\proba_{[K\,\vx_{0}]}\bigg(\sum_{j=0}^{[\MK(A)\,s]}X_{j}(A,K)=q\,;\,
\TK_{A}>\deltaK\bigg) +\deltaK\,e^{-C(\vx_{0},G)\,K}\;,
\end{align*}
where the last term tends to zero when $K$ tends to infinity.

Let
$$
X'_{0}(A,\,K)=\begin{cases}1&\;\mathrm{if}\; 
\NK_{t}\in K*A^c,\,\mathrm{for\ some}\; t\in [\deltaK,\,\etaK(A)[\;,\\
0&\;\mathrm{otherwise}\;.
\end{cases}
$$
We have by Corollary \ref{iter:Kurtz}
\begin{align*}
&\proba_{[K\,\vx_{0}]}\bigg(\sum_{j=0}^{[\MK(A)\,s]}X_{j}(A,K)=q\,;\,
\TK_{A}>\deltaK\bigg)
\\&
=\proba_{[K\,\vx_{0}]}\bigg(X'_{0}(A,K)+\sum_{j=1}^{[\MK(A)\,s]}X_{j}(A,K)=q\,;\,
\TK_{A}>\deltaK\bigg)
\\&
=
\proba_{[K\,\vx_{0}]}\bigg(X'_{0}(A,K)+\sum_{j=1}^{[\MK(A)\,s]}X_{j}(A,K)=q
\bigg)+\mathcal{O}\left(\deltaK\,e^{-C(\vx_{0},G)\,K}\right)\;.
\end{align*}
Using the Markov property we have
\begin{align*}
&\proba_{[K\,\vx_{0}]}\bigg(X'_{0}(A,K)+\sum_{j=1}^{[\MK(A)\,s]}X_{j}(A,K)=q\,\bigg)=\esperance_{[K\,\vx_{0}]}\left(\esperance_{\vN(\deltaK)}\left(
\un_{X'_{0}(A,K)+\sum_{j=1}^{[\MK(A)\,s]}X_{j}(A,K)=q}
\circ\theta_{-\deltaK}\right)\right)\;,
\end{align*}
where we recall that $\theta_{-\deltaK}$ is the $-\deltaK$ time translation.

\smallskip Using \eqref{melange}  with $t= \deltaK$, we get
\begin{align*}
&\bigg|\proba_{[K\,\vx_{0}]}\bigg(X'_{0}(A,K)+\sum_{j=1}^{[\MK(A)\,s]}X_{j}(A,K)=q
\bigg)
\\&
\hskip 1cm -e^{-\lambK\,\deltaK}\lauK([K\,\vx_{0}])\,
\esperance_{\laKqsd}\left(
\un_{X'_{0}(A,K)+\sum_{j=1}^{[\MK(A)\,s]}X_{j}(A,K)=q}
\circ\theta_{-\deltaK}\right)\bigg|
\\&
\le
\Gamma(K,\,\deltaK)+\bigg(1-e^{-\lambK\,\deltaK}\lauK([K\,\vx_{0}])\bigg)
\, \esperance_{\delta_{\vecz}}\left(
\un_{X'_{0}(A,K)+\sum_{j=1}^{[\MK(A)\,s]}X_{j}(A,K)=q}
\circ\theta_{-\deltaK}\right)
\bigg)\;.
\end{align*}
By Lemma \ref{conv:uKnf}, \eqref{choix:deltaK} and \eqref{cherlambda},
it follows that the right hand side tends to zero when $K$ tends to
infinity.

We similarly  conclude that
\begin{align*}
&\lim_{K\to\infty}
\bigg|\proba_{[K\,\vx_{0}]}\bigg(X'_{0}(A,K)+\sum_{j=1}^{[\MK(A)\,s]}X_{j}(A,K)=q
\bigg)
-\esperance_{\laKqsd}\left(
\un_{X'_{0}(A,K)+\sum_{j=1}^{[\MK(A)\,s]}X_{j}(A,K)=q}
\circ\theta_{-\deltaK}\right)\bigg|=0\;.
\end{align*}
Using \eqref{invqsd} we have
\begin{align*}
&\esperance_{\laKqsd}\left(
\un_{X'_{0}(A,K)+\sum_{j=1}^{[\MK(A)\,s]}X_{j}(A,K)=q}
\circ\theta_{-\deltaK}\right)=e^{\lambK \deltaK}
\proba_{\laKqsd}\bigg(X'_{0}(A,K)+\sum_{j=1}^{[\MK(A)\,s]}X_{j}(A,K)=q
\bigg)
\end{align*}
and using Lemma \ref{borne:Z0},
\begin{align*}
&\bigg|\proba_{\laKqsd}\bigg(X'_{0}(A,K)+\sum_{j=1}^{[\MK(A)\,s]}X_{j}(A,K)=q
\bigg)-\proba_{\laKqsd}\bigg(\sum_{j=0}^{[\MK(A)\,s]}X_{j}(A,K)=q
\bigg)\bigg|
\\&
\le \proba_{\laKqsd}\big(\TK_{A}< \deltaK\big)=
\proba_{\laKqsd}\big(Z_{0}(A,K)=1\big)\le e^{-\beta_{A}\,K}\;.
\end{align*}
Collecting all the error terms we get
$$
\lim_{K\to\infty}\bigg|
\proba_{K\vx_{0}}\bigg(\sum_{j=0}^{[\MK(A)\,s]}X_{j}(A,K)=q
\bigg)-\proba_{\laKqsd}\bigg(\sum_{j=0}^{[\MK(A)\,s]}X_{j}(A,K)=q
\bigg)\bigg|=0
$$
and the result follows.
\end{proof}



\section{Exponential law and some related results.}
\label{sec:4}

\begin{theorem}\label{loiexpo}
Let $A$ be an open subset of $\rpd$ satisfying Assumption ${\bf (HS)}$.
For any $t\ge0$
$$
\lim_{K\to\infty}\proba_{\laKqsd}\big(\TK_{A}>t\, \MK(A)\, \etaK(A)\big)=e^{-t}\;.
$$
In other words, the sequence  $\Big(\frac{\TK_{A}}{ \MK(A)\, \etaK(A)}\Big)$ issued from $\laKqsd$ converges in law to the exponential law of parameter $1$, as $K$ tends to infinity.
\end{theorem}
\begin{proof}
We first observe that since the $(X_{j}(A,K))$ are Bernoulli random
variables we have  for any $q\in\mathds{N}$
\begin{align*}
\proba_{\laKqsd}\big(\TK_{A}>(q+1) \etaK(A)\big)&=
\proba_{\laKqsd}\big(X_{j}(A,K)=0\,,\, 0\le j\le q\big)=\proba_{\laKqsd}\bigg(\sum_{j=0}^{q}X_{j}(A,K)=0\bigg)\;.
\end{align*}
From Theorem  \ref{MainBD}  we conclude that for any $t\ge0$
$$
\lim_{K\to\infty}\proba_{\laKqsd}\big(\TK_{A}>   ([t\,\MK(A)]+1)   \,
\etaK(A)\big)=e^{-t}\;.
$$

Thanks to  the 
divergence of $\MK(A) $ with $K$ (see equality \eqref{choix:MKA}), for any $\varepsilon$ small enough and for any  $K$ large enough, we have that
$$[(t-\varepsilon)\,\MK(A)] + 1\le t\,\MK(A) \le [(t+\varepsilon)\,\MK(A)] + 1.$$
Then,
\begin{align*}
\proba_{\laKqsd}\big(\TK_{A}>   ([(t+\varepsilon)\,\MK(A)]+1)   \,
\etaK(A)\big)& \le \proba_{\laKqsd}\big(\TK_{A}>t\, \MK(A)\, \etaK(A)\big)\\
&\le \proba_{\laKqsd}\big(\TK_{A}>([(t-\varepsilon)\,\MK(A)]+1)\, \etaK(A)\big).
\end{align*}
Therefore, we deduce that 
 $$e^{-(t+\varepsilon)}
\le  \lim_{K\to\infty}\proba_{\laKqsd}\big(\TK_{A}>t\, \MK(A)\, \etaK(A)\big)\le  \lim_{K\to\infty}e^{-(t-\varepsilon)} ,$$
and the result follows by letting $\varepsilon$ tend to $0$. 
\end{proof}

\begin{theorem}
\label{loiTAK}  
The sequence  of random variables $\frac{\TK_{A}}{ \MK(A)\, \etaK(A)}$ issued from $\laKqsd$ is exponentially tight. It follows that all moments converge to the moments of the exponential law of parameter $1$.

In particular,  $$
\lim_{K\to\infty}\frac{ \esperance_{\laKqsd}(\TK_{A}) }{ \MK(A)\, \etaK(A) }=1\;,
$$
and
$$
\lim_{K\to\infty}\proba_{\laKqsd}\big(\TK_{A}>t\,  \esperance_{\laKqsd}(\TK_{A}) \big)=e^{-t}\;.
$$

 \end{theorem}
 
\begin{proof}

Let for $q\in \mathds{N}$
$$
\Theta(q,K)=\proba_{\laKqsd}\big(\TK_{A}>  q\,[\MK(A) +1]\,  \,
\etaK(A)\big)\;.
$$
As observed in the previous proof we have
\begin{align*}
\Theta(q+1,K)&=\proba_{\laKqsd}\bigg(\sum_{j=0}^{   (q+1)\,[\MK(A) +1]-1\,  \,
   }X_{j}(A,K)=0\bigg)
\\&
\le 
\proba_{\laKqsd}\bigg(\sum_{j=0}^{   q\,(\MK(A) +1)-1\,     }X_{j}(A,K)=0\,;\,  
\sum_{j=   q\,(\MK(A) +1) +1\,      }^{   (q+1)\,(\MK(A) +1)-1\,     }
X_{j}(A,K)=0
\bigg).
\end{align*}
We apply Theorem \ref{RBD} and  the bound \eqref{supuK} with $t_{2}= (q+1)\,(\MK(A) +1)\,\etaK(A) $, $t_{1}= q (\MK(A) +1)\,\etaK(A) $ and $s=\etaK(A) $, to obtain

\begin{align*}
&\proba_{\laKqsd}\bigg(\sum_{j=0}^{   q\,[\MK(A) +1]-1\,     }X_{j}(A,K)=0\,;\,  
\sum_{j=   q\,(\MK(A) +1) +1\,      }^{   (q+1)\,(\MK(A) +1-1\,     }
X_{j}(A,K)=0\bigg)\\
&\le \Theta(q,K)\,\Gamma\big(\etaK(A)  ,K\big)
 + (1+e^{-C'' K})\Theta(q,K) 
\proba_{\laKqsd}\bigg(\sum_{j=   q\,(\MK(A) +1)+1\,    }^{   (q+1)\,(\MK(A) +1)-1 }X_{j}(A,K)=0 \bigg).
\end{align*}

Since 
$$
\sum_{j=   q\,(\MK(A) +1)+1\,     }^{   (q+1)\,(\MK(A) +1)-1\,      }
X_{j}(A,K)=0
$$
implies
$$
\tauex\ge (q+1)\,(\MK(A) +1)\; \etaK(A) ,
$$
 we get
\begin{align*}
&\proba_{\laKqsd}\left(\sum_{j=   q\,(\MK(A) +1)+1\,    }^{   (q+1)\,(\MK(A) +1)-1 }X_{j}(A,K)=0 \right)\\
&=\proba_{\laKqsd}\left(\sum_{j=   q\,(\MK(A) +1)+1\,     }^{   (q+1)\,(\MK(A) +1)-1\,      }
X_{j}(A,K)=0\,;\, \tauex\ge (q+1)\,(\MK(A) +1)\; \etaK(A)  \right)\\
&= e^{-\lambK     q\,(\MK(A) +1)\,  \,\etaK(A)  }\,
\proba_{\laKqsd}\left(\sum_{j= 0  }^{   \MK(A)  -1     }
X_{j}(A,K)=0
\right)\\
&= e^{-\lambK     q\,(\MK(A) +1)\,  \,\etaK(A)  }\,
\proba_{\laKqsd}\left(\TK_{A}>  \MK(A)  \,
\etaK(A)\right)\;. 
\end{align*}

Therefore
$$
\Theta(q+1,K) \le \Theta(q,K) \, \left\{\Gamma\big(\etaK(A)  ,K\big)
 + (1+e^{-C'' K}) \proba_{\laKqsd}\left(\TK_{A}>  \MK(A)  \,
\etaK(A)\right)\right\}.
$$
Using  \eqref{borneGamma}, Theorem \ref{loiexpo} and the definition of $\etaK(A)$, we get

$$
\frac{\Theta(q+1,K) }{\Theta(q,K) }\le \left(\frac{1}{e}+o(1)\right)< \frac{1}{2} 
$$
for $K$ large enough. This implies that  for any $K$ large enough, for any $q$
$$
 \Theta(q,K) \le 2^{-q}.
$$
The uniform exponential tightness of the random variables $(\tfrac{\TK_{A}}{ \MK(A)\, \etaK(A)})$ issued from $\laKqsd$ and their uniform integrability follow. 
\end{proof}

\medskip
 We  now relate the time scale of the exit time $\TK_{A}$ under
the initial distribution $\laKqsd$  and the large deviations potential.

\begin{theorem}\label{echelleTAK}
Under the hypotheses {\bf (HV)},
 we have for any $\epsilon>0$
and any open set $A$ of $\rpd$ satisfying Assumption ${\bf (HS)}$, that
$$
\lim_{K\to\infty}\proba_{ \laKqsd}\left(e^{K\,(V_{*}(A)-\epsilon)}\le
  \TK_{A}\le e^{K\,(V_{*}(A)+\epsilon)}\right)=1\;.
$$
\end{theorem}

This result is well known if the initial distribution is a Dirac measure (see references in the proof below). The difficulty here is to obtain the same result starting from a qsd. 
\begin{proof}
Recall that $\,
\deltaK = [K^2] < e^{K\,(V_{*}(A)-\epsilon/2)}
$, for $K$ large enough.  By the Markov property we have  by \eqref{melange} for $K$ large enough  that
\begin{align*}
&\proba_{ \nf}\left(e^{K\,(V_{*}(A)-\epsilon/2)}\le
  \TK_{A}\le e^{K\,(V_{*}(A)+\epsilon/2)}\right)\\
&\le 
\esperance_{\nf}\left(
\esperance_{\vN(\deltaK)}\left(\un_{\exp(K\,(V_{*}(A)-\epsilon/2))-\deltaK\le
  \TK_{A}\le \exp(K\,(V_{*}(A)+\epsilon/2))-\deltaK}
\right)\right)\\
&\le e^{-\lambK\,\deltaK}\,\lauK(\nf)\,  \proba_{\laKqsd}\big(
\exp(K\,(V_{*}(A)-\epsilon/2))-\deltaK\le
  \TK_{A}\le \exp(K\,(V_{*}(A)+\epsilon/2))-\deltaK \big)\\
&\hskip 0.3cm +\big(1- e^{-\lambK \deltaK}\,\lauK(\nf) \big)
+\Gamma(\deltaK,\,K)\;.
\end{align*}
Let $K$ be large enough so that
$$
e^{K\,(V_{*}(A)-\epsilon)}<e^{K\,(V_{*}(A)-\epsilon/2)}-\deltaK\, $$
and $$
e^{K\,(V_{*}(A)+\epsilon)}>e^{K\,(V_{*}(A)+\epsilon/2)}-\deltaK\;.
$$
We have 
\begin{align*}
1&\ge \proba_{ \laKqsd}\left(e^{K\,(V_{*}(A)-\epsilon)}\le
  \TK_{A}\le e^{K\,(V_{*}(A)+\epsilon)}\right)
\\&
\ge \proba_{\laKqsd}\big(
\exp(K\,(V_{*}(A)-\epsilon/2))-\deltaK\le
  \TK_{A}\le \exp(K\,(V_{*}(A)+\epsilon/2))-\deltaK\big)
\\&
\ge e^{\lambK\,\deltaK}\,\frac{1}{\lauK(\nf)}
\;\proba_{ \nf}\left(e^{K\,(V_{*}(A)-\epsilon/2)}\le
  \TK_{A}\le e^{K\,(V_{*}(A)+\epsilon/2)}\right)
\\&\hskip 0.3cm
-e^{\lambK\,\deltaK}\,\frac{1}{\lauK(\nf)}\big(1- e^{-\lambK \deltaK}\,\lauK(\nf) \big)-
e^{\lambK\,\deltaK}\,\frac{1}{\lauK(\nf)}
\;\Gamma(\deltaK,\,K)\;.
\end{align*}
From Theorem 6.17 (i) page 136 in \cite{SW} or Theorem 7.8  in \cite{KP} we have for
any $\epsilon/2>0$
$$
\lim_{K\to\infty}\proba_{ \nf}\left(e^{K\,(V_{*}(A)-\epsilon/2)}\le
  \TK_{A}\le e^{K\,(V_{*}(A)+\epsilon/2)}\right)=1\;,
$$
and the result follows using Lemma \ref{lambdaV} and Theorem
\ref{ordreV} since from our choice of $\deltaK$ (see
\eqref{choix:deltaK} and \eqref{cherlambda}), we have
$$
\lim_{K\to\infty} e^{\lambK\,\deltaK}=1
$$
and from \ref{conv:uKnf}
$$
\lim_{K\to\infty} \lauK(\nf)=1\;.
$$
\end{proof}

\begin{theorem}\label{equivtout}
Under Assumptions {\bf (HV)},
 we have for any
open  set $A$ of $\, \rpd$ satisfying Assumption ${\bf (HS)}$,
$$
\lim_{K\to\infty}\frac{\log \big(\esperance_{\laKqsd}(\TK_{A})\big)
}{K}=\lim_{K\to\infty}\frac{\log\big(\etaK(A)\;\MK(A)\big) }{K}=
V_{*}(A)\;.
$$
\end{theorem}
\begin{proof}
From Theorem \ref{loiTAK} it is enough to prove the second equality. 

Let $\epsilon>0$, for $K$ large enough we have from 
Theorem \ref{echelleTAK}
$$
\proba_{ \laKqsd}\left(e^{K\,(V_{*}(A)-\epsilon)}\le
  \TK_{A}\le e^{K\,(V_{*}(A)+\epsilon)}\right)\ge \frac{2}{3}\;.
$$
From Theorem \ref{loiexpo}, we have for $K$ large enough 
$$
\proba_{ \laKqsd}\left(\frac{\etaK(A)\;\MK(A)}{4}\le
  \TK_{A}\le 2\, \etaK(A)\;\MK(A)\right)\ge \frac{2}{3}\;.
$$
Therefore with $\proba_{ \laKqsd}$ probability larger than $1/3$ we
have
$$
e^{K\,(V_{*}(A)-\epsilon)}\le
  \TK_{A}\le e^{K\,(V_{*}(A)+\epsilon)}\;,
$$
and
$$
\frac{\etaK(A)\;\MK(A)}{4}\le
  \TK_{A}\le 2\, \etaK(A)\;\MK(A)\;.
$$
This implies
$$
e^{K\,(V_{*}(A)-\epsilon)}\le 2\, \etaK(A)\;\MK(A)
$$
and
$$
\frac{\etaK(A)\;\MK(A)}{4}\le e^{K\,(V_{*}(A)+\epsilon)}\;.
$$
We conclude that
\begin{align*}
V_{*}(A)-\epsilon&\le\liminf_{K\to\infty}\frac{\log\big(\etaK(A)\;\MK(A)\big)
}{K} 
\\&\le \limsup_{K\to\infty}\frac{\log\big(\etaK(A)\;\MK(A)\big)}{K} 
\le V_{*}(A)+\epsilon\;.
\end{align*}
Since this is true for any $\epsilon>0$ the result follows.
\end{proof}

\medskip 
We are now interested in ergodic properties of the process starting from $\nu_{K}$. 

\begin{lemma}\label{nuAcL}
For any open subset $A$ of $\rpd$ satisfying Assumption {\bf (HS)} and any $\lambK^{-1}>L>
\deltaK$, there exist  constants $C_{1}(A)$ and $C_{2}(A)$ such that   for any $0\le s<L-\,\deltaK$
$$
\left|\esperance_{\laKqsd}\left(\un_{K*A^{c}}\big(\vN(s)\big)\; 
\un_{\tauex>L}\right)-e^{-\lambK\,L}\;\mu_{\sK}\big(K*A^{c}\big)\right|\le
C_{1}(A)\,\lambK\,\deltaK\;.
$$
If $L-\,\deltaK\le s\le L$ we have
$$
\esperance_{\laKqsd}\left(\un_{K*A^{c}}\big(\vN(s)\big)\; 
\un_{\tauex>L}\right)\le \laKqsd\big(K*A^{c}\big)\le 
C_{2}(A) \; \mu_{\sK}\big(K*A^{c}\big)\;.
$$
\end{lemma}

\medskip
\begin{proof}
For $0\le s<L-\,\deltaK$, from \eqref{invqsd} and the
 Markov property we have
\begin{align*}
&\esperance_{\laKqsd}\left(\un_{K*A^{c}}\big(\vN(s)\big)\; 
\un_{\tauex>L}\right)=e^{-\lambK\,s}\;\esperance_{\laKqsd}
\left(\un_{K*A^{c}}\big(\vN(0)\big)\; \un_{\tauex>L-s}\right)
\\&
=e^{-\lambK\,s}\;\esperance_{\laKqsd}
\left(\un_{K*A^{c}}\big(\vN(0)\big)\;\un_{\tauex>\,\deltaK}
\esperance_{\vN(\,\deltaK)}
\left( \un_{\tauex>L-s-\deltaK}\right)\right)
\\&
=e^{-\lambK\,s}\;\esperance_{\laKqsd}
\left(\un_{K*A^{c}}\big(\vN(0)\big)\;
\esperance_{\vN(\,\deltaK)}
\left( \un_{\tauex>L-s-\,\deltaK}\right)\right)
+\mathcal{O}\big( \laKqsd\big(\tauex\le \,\deltaK\big)\big)\;.
\end{align*}
Note that 
\begin{align*}
&\esperance_{\laKqsd}
\left(\un_{K*A^{c}}\big(\vN(0)\big)\;
\esperance_{\vN(\,\deltaK)}
\left( \un_{\tauex>L-s-\,\deltaK}\right)\right)\\
& = \sum_{m}\nu_{K}(m)\un_{K*A^{c}}(m) \sum_{n} \proba_{\vm}\big(\vN(
\deltaK)=\vn\big) \proba_{\vn}(\tauex>L-s-\,\deltaK).
\end{align*}
Then, using Corollary \ref{mieuxcor} (ii) and \eqref{invqsd} we get
$$
\Bigg|\esperance_{\laKqsd}
\left(\un_{K*A^{c}}\big(\vN(0)\big)\;
\esperance_{\vN(\,\deltaK)}
\left( \un_{\tauex>L-s-\,\deltaK}\right)\right)
-\mu_{\sK}\big(K*A^{c}\big)\;e^{-\lambK\,(L-s)}\Bigg|
\le \lambK\;.
$$
The first bound  follows using \eqref{expqsd}. Indeed,
\begin{align*}
&\left|\esperance_{\laKqsd}\left(\un_{K*A^{c}}\big(\vN(s)\big)\; 
\un_{\tauex>L}\right)-e^{-\lambK\,L}\;\mu_{\sK}\big(K*A^{c}\big)\right| \\
& \le  \left| e^{-\lambK s}\;\esperance_{\laKqsd}
\bigg(\un_{K*A^{c}}\big(\vN(0)\big)\;
\esperance_{\vN(\,\deltaK)}
\left( \un_{\tauex>L-s-\,\deltaK}\right)\right)\\
&\hskip 1cm
+\mathcal{O}\big( \laKqsd\big(\tauex\le \,\deltaK\big)\big) -e^{-\lambK\,L}\;\mu_{\sK}\big(K*A^{c}\big)\bigg|\\
&\le  \left| e^{-\lambK s}\;\esperance_{\laKqsd}
\bigg(\un_{K*A^{c}}\big(\vN(0)\big)\;
\esperance_{\vN(\,\deltaK)}
\left( \un_{\tauex>L-s-\,\deltaK}\right)\right) -e^{-\lambK\,L}\;\mu_{\sK}\big(K*A^{c}\big)\bigg|\\
&\hskip 1cm + |\mathcal{O}\big( \laKqsd\big(\tauex\le \,\deltaK\big)\big)| \\
&\le \lambK + C(A)(1- e^{-\lambK \deltaK}) \qquad (C(A) \hbox{ being a positive constant})\\
&\le C_{1}(A) \lambK \deltaK\,,
\end{align*} 
by using \eqref{expqsd} and since $\lambK \deltaK$ tends to $0$ as $K$ tends to infinity. 

\noindent The second bound follows from \eqref{invqsd} and \eqref{supuK}.

\end{proof}

We now prove a mean ergodic like result  starting from $\nu_{K}$.  Recall that $\nu_{K}$ is not an invariant measure and we will see in the next lemma that for a time large enough but not too large to avoid extinction, the time average approaches the 
measure $\mu_{K}$.

\begin{lemma}\label{tmpAc}
If
$A$  is an open subset of $\rpd$ satisfying Assumption ${\bf (HS)}$, there exists a constant $C_{3}(A)$ such that  for any $L$ with $\lambK^{-1}>L> \,\deltaK$ and
for $K$  large enough,
\begin{align}
\label{56-1}
&\left|\frac{1}{L}\;\esperance_{\laKqsd}\left(
\un_{\tauex>L}\;
\int_{0}^{L}\un_{K*A^{c}}\big(\vN(s)\big)\;ds\right)-\mu_{K}(K*A^{c})\right|
\le C_{3}(A) \Big((\lambK\,L+\,\deltaK/L)\;\mu_{\sK}\big(K*A^{c}\big)\;+
\lambK\,\deltaK\Big)\;.
\end{align}
We also have
\begin{align}
\label{56-2}
&
\frac{\frac{1}{L}\;\esperance_{\laKqsd}\left(
\un_{\tauex>L}\;
\int_{0}^{L}\un_{K*A^{c}}\big(\vN(s)\big)\;ds\right)
- C_{3}(A)\lambK\,\,\deltaK}{1+ C_{3}(A)(\lambK\,L+\,\deltaK/L)}\nonumber \\
&
\le \mu_{\sK}\big(K*A^{c}\big)\le
\frac{\frac{1}{L}\;\esperance_{\laKqsd}\left(
\un_{\tauex>L}\;
\int_{0}^{L}\un_{K*A^{c}}\big(\vN(s)\big)\;ds\right)
+C_{3}(A) (\lambK\,\,\deltaK)}{1-C_{3}(A)(\lambK\,L+\,\deltaK/L)}\;.
\end{align}
\end{lemma}

\medskip
\begin{proof}
We have by Fubini's Theorem
$$
\esperance_{\laKqsd}\left(
\un_{\tauex>L}\;
\int_{0}^{L}\un_{K*A^{c}}\big(\vN(s)\big)\;ds\right)=
\int_{0}^{L}ds\;\esperance_{\laKqsd}\left(\un_{\tauex>L}\;
\un_{K*A^{c}}\big(\vN(s)\big)\right).
$$
Then we have
\begin{align*}
&\left|\frac{1}{L}\;\esperance_{\laKqsd}\left(
\un_{\tauex>L}\;
\int_{0}^{L}\un_{K*A^{c}}\big(\vN(s)\big)\;ds\right)-\mu_{K}(K*A^{c})\right|\\
& \le \left|\frac{1}{L}\;
\int_{0}^{L-\deltaK}\esperance_{\laKqsd}\left(
\un_{\tauex>L}\;
\un_{K*A^{c}}\big(\vN(s)\big)\;ds\right) ds - \frac{L-\deltaK}{L}e^{-\lambK L}\mu_{K}(K*A^{c})\right| \\
&\hskip 1cm  + \frac{L-\deltaK}{L}(1 -e^{-\lambK L})\mu_{K}(K*A^{c})\\
&\hskip 1cm   +  \left|\frac{1}{L}\;
\int_{L-\deltaK}^{L}\esperance_{\laKqsd}\left(
\un_{\tauex>L}\;
\un_{K*A^{c}}\big(\vN(s)\big)\;ds\right) ds - \frac{\deltaK}{L}\mu_{K}(K*A^{c})\right|\\
&\le  C_{1}(A) \frac{L-\deltaK}{L} \lambK\,\deltaK + \lambK\, L\,\mu_{K}(K*A^{c}) + (C_{2}(A)+1) \frac{\deltaK}{L}\, \,\mu_{K}(K*A^{c})
\end{align*}
where $C_{1}(A)$  and $C_{2}(A)$ have been introduced in Lemma \ref{nuAcL}.
\end{proof}

\medskip  We have established in Lemma \ref{largedevW}   an upper-bound for $\mu_{K}(K*A^{c})$. In the next theorem, we give a more precise asymptotics for $\mu_{K}(K*A^{c})$ and relate this quantity  to the inverse of the scale of the exit time from $A$. A similar result holds for  $\nu_{K}(K*A^{c})$.
\begin{theorem}\label{gdvAc}
Let $A$  be an open subset of $\rpd$ satisfying Assumption ${\bf (HS)}$. Then there exists a number $\rho_{A}>0$ independent of
$K$ such that for $K$ large enough
$$
\frac{\rho_{A}}{3\,K^2\,\MK(A)\,\etaK(A) }\le
\mu_{K}(K*A^{c})\le\frac{2\;\deltaK}{\MK(A)\,\etaK(A)}\;.
$$
Moreover,
$$
\lim_{K\to\infty}\frac{\log\big(\mu_{\sK}(K*A^{c})\big)}{K}
=\lim_{K\to\infty}\frac{\log\big(\laKqsd(K*A^{c})\big)}{K}=-V^{*}(A)\;.
$$
\end{theorem}

\medskip
\begin{proof}
The  proof of  the large deviation principle for an invariant measure in  \cite{SW} and in  \cite{VF} requires some assumptions that  do not hold as such in our case.
Therefore we cannot directly apply their results to the q-process, neither directly to 
 the
\qsd

\medskip The proof extensively uses Lemma \ref{tmpAc} with $L=\MK(A)\,\etaK(A)$. 
We have
\begin{align}
\label{occupation}
&\esperance_{\laKqsd}\left(
\un_{\tauex>\MK(A)\,\etaK(A)}\;
\int_{0}^{\MK(A)\,\etaK(A)}\un_{K*A^{c}}\big(\vN(s)\big)\;ds\right)\nonumber\\
&=\sum_{j=0}^{\MK(A)-1}\esperance_{\laKqsd}\left(
\un_{\tauex>\MK(A)\,\etaK(A)}\;
\int_{j\,\etaK(A)}^{(j+1)\,\etaK(A)}\un_{K*A^{c}}\big(\vN(s)\big)\;ds\right).
\end{align}
From the properties of the \qsd we get for any $0\le j\le\MK(A)-1 $
\begin{align*}
&\esperance_{\laKqsd}\left(
\un_{\tauex>\MK(A)\,\etaK(A)}\;
\int_{j\,\etaK(A)}^{(j+1)\,\etaK(A)}\un_{K*A^{c}}\big(\vN(s)\big)\;ds\right)\\
&=
e^{-\lambK\,j\,\etaK(A)}
\esperance_{\laKqsd}\left(
\un_{\tauex>(\MK(A)-j)\,\etaK(A)}\;
\int_{0}^{\etaK(A)}\un_{K*A^{c}}\big(\vN(s)\big)\;ds\right)\;.
\end{align*}
We now observe that if $X_{0}(A,K)=0$ we have
$$
\int_{0}^{\etaK(A)}\un_{K*A^{c}}\big(\vN(s)\big)\;ds=0\;.
$$
Therefore we have to obtain an upper and a lower bound for 
\begin{equation}
\label{occ}
\esperance_{\laKqsd}\left(\un_{X_{0}(A,K)=1}\;
\un_{\tauex>(\MK(A)-j)\,\etaK(A)}\;
\int_{0}^{\etaK(A)}\un_{K*A^{c}}\big(\vN(s)\big)\;ds\right),
\end{equation}
with $0\le j\le \MK(A)-1$.

Note that this quantity is increasing in $j$. Therefore, we will only consider $j=0$ for the lower bound and $j= \MK(A) -1$  and $j= \MK(A) -2$ for the upper bound.

For the lower bound, using the Markov property 
(and $\tauex\ge\TK_{A}$) we get
\begin{align*}
&\esperance_{\laKqsd}\left(\un_{X_{0}(A,K)=1}\;
\un_{\tauex>\MK(A)\, \,\etaK(A)}\;
\int_{0}^{\etaK(A)}\un_{K*A^{c}}\big(\vN(s)\big)\;ds\right)
\\&
=\esperance_{\laKqsd}\left(\un_{\TK_{A}<\etaK(A)}\;\esperance_{\vN(\TK_{A})}\left(
\un_{\tauex>\MK(A) \,\etaK(A)-\TK_{A}}\;\ \int_{0}^{\etaK(A)-\TK_{A}}\un_{K*A^{c}}\big(\vN(s)\big)\;ds\right)\right).
\end{align*}
Since the jump rate is bounded below on the exterior boundary $\partial_{e}(K*A)$ (the set of points in $K*A^c$ at distance $1$ from $K*A$), we deduce
that there exists a constant $\rho_{A}>0$ such that for any $K$
large enough
$$
\inf_{\vn\in \partial_{e}(K*A)}\proba_{\vn}\big(\vN(t)=
\vn\,,\;\forall\;t\in[0,1/K]\big)>\rho_{A}/K\;.
$$
Therefore
\begin{align*}
&\esperance_{\laKqsd}\left(\un_{\TK_{A}<\etaK(A)-1/K}
\un_{\tauex>\MK(A)\,\etaK(A)}\;
\int_{0}^{\etaK(A)}\un_{K*A^{c}}\big(\vN(s)\big)\;ds\right)\\
&\ge 
\esperance_{\laKqsd}\bigg(\un_{\TK_{A}<\etaK(A)-1/K} \un_{\{\vN(t) = \vN(\TK_{A}); \, \TK_{A}\le t\le \TK_{A}+\frac{1}{K}\}}
 \un_{\{\tauex>\MK(A)\,\etaK(A)\}}\;
\int_{\TK_{A}}^{\TK_{A}+1/K}\un_{K*A^{c}}\big(\vN(s)\big)\;ds\bigg)\\
&\ge \frac{\rho_{A}}{K^2}\;\proba_{\laKqsd}(\TK_{A}<\etaK(A)-1/K)\;.
\end{align*}
Thus we get
\begin{align*}
&\esperance_{\laKqsd}\left(
\un_{\tauex>\MK(A)\,\etaK(A)}\;
\int_{0}^{\MK(A)\,\etaK(A)}\un_{K*A^{c}}\big(\vN(s)\big)\;ds\right)\\
&\ge e^{-\lambK\,\MK(A)\,\etaK(A)}\MK(A)\; \frac{\rho_{A}}{K^2}\,\proba_{\laKqsd}(\TK_{A}<\etaK(A)-1/K).
\end{align*}
We also have using \eqref{invqsd} that
\begin{align*}
&\proba_{\laKqsd}\big(\etaK(A)-1/K\le \TK_{A}<\etaK(A)\big)\\
&\le \lambK\etaK(A)+
\proba_{\laKqsd}\big(\etaK(A)-1/K\le \TK_{A}<\etaK(A)\,,\; \tauex>\etaK(A)
\big)
\\ &
\le \lambK\etaK(A)+
\proba_{\laKqsd}\big(0\le \TK_{A}<1/K\,,\; \tauex>1/K
\big)
\le \lambK\etaK(A)+ \proba_{\laKqsd}\big(Z_{0}(A,K)=1\big)
\\&
\le \lambK\etaK(A)+e^{-\beta_{A}\,\sK}
\end{align*}
by Lemma \ref{borne:Z0} (Recall that $\deltaK=[K^2]$).\\

 Summing over $j$ in \eqref{occupation}, we obtain
\begin{align*}
&\esperance_{\laKqsd}\left(
\un_{\tauex>\MK(A)\,\etaK(A)}\;
\int_{0}^{\MK(A)\,\etaK(A)}\un_{K*A^{c}}\big(\vN(s)\big)\;ds\right)\\
& \ge e^{-\lambK\,\MK(A)\,\etaK(A)}\MK(A)\; \frac{\rho_{A}}{K^2}\,
\left(\;\proba_{\laKqsd}(X_{0}(A,K)=1\big)
-\,\lambK\etaK(A)-e^{-\beta_{A}\,\sK}\right)\ge \frac{\rho_{A}}{2\,K^2}
\end{align*}
for $K$ large enough using $\alpha_{A}<\beta_{A}$, Lemma
\ref{encadre} and \eqref{asympto}.\\

Using Lemma \ref{tmpAc} with $L=\MK(A)\,\etaK(A)$, we get for $K$ large enough
$$
\mu_{K}(K*A^{c})\ge
\frac{\rho_{A}/(2\,K^2\,\MK(A)\,\etaK(A))
- C_{3}\,\lambK\,\,\deltaK}{1+ C_{3}(\lambK\,\MK(A)\,\etaK(A)
+\deltaK/\MK(A)\,\etaK(A))}
 \ge \frac{\rho_{A}}{3\,K^2\,\MK(A)\,\etaK(A)}\;.
$$

\bigskip
For the upper bound in \eqref{occ},we use  Lemma \ref{tmpAc} with $L=\etaK(A)$, and first consider the largest term with $j=\MK(A)-1$.
\begin{align*}
&\esperance_{\laKqsd}\left(\un_{X_{0}(A,K)=1}\;
\un_{\tauex>\,\etaK(A)}\;
\int_{0}^{\etaK(A)}\un_{K*A^{c}}\big(\vN(s)\big)\;ds\right)\\
&\le \esperance_{\laKqsd}\left(
\un_{\tauex>\,\etaK(A)}\;
\int_{0}^{\etaK(A)}\un_{K*A^{c}}\big(\vN(s)\big)\;ds\right)\\
&
\le \etaK(A)\;\big(\big(1+C_{3}(\lambK\,\etaK(A)+\,\deltaK/\etaK(A))\big)\;
\mu_{\sK}\big(K*A^{c}\big)\;+
C_{3}\,\lambK\,\deltaK\big).
\end{align*}
The above upper bound  is too large to provide the good estimate. We need a more precise upper bound for the other terms ($j\le \MK(A)-2$), using the mixing property.

\medskip 
From the monotonicity, it is enough to consider the case  $j= \MK(A)-2$. We have
\begin{align*}
&\esperance_{\laKqsd}\left(\un_{X_{0}(A,K)=1}\;
\un_{\tauex>2\,\etaK(A)}\;
\int_{0}^{\etaK(A)}\un_{K*A^{c}}\big(\vN(s)\big)\;ds\right)\\
&=\esperance_{\laKqsd}\Bigg(\un_{\TK_{A}<\etaK(A)}\;\esperance_{\vN(\TK_{A})}
\left(\un_{\tauex>2\,\etaK(A)-\TK_{A}}\;
\int_{0}^{\etaK(A)-\TK_{A}}\un_{K*A^{c}}\big(\vN(s)\big)\;ds\right)\Bigg)\;.
\end{align*}

We will estimate from above 
  for each $\vn\in \partial_{e}(K*A)$ and
 each $0\le u<\etaK(A)$ the quantity
$$
\esperance_{\vn}\left(
\un_{\tauex>2\,\etaK(A)-u}\;
\int_{0}^{\etaK(A)-u}\un_{K*A^{c}}\big(\vN(s)\big)\;ds\right)\;.
$$
If $\etaK(A)-2\,\deltaK\le u\le \etaK(A)$  we have
$$
\esperance_{\vn}\left(
\un_{\tauex>2\,\etaK(A)-u}\;
\int_{0}^{\etaK(A)-u}\un_{K*A^{c}}\big(\vN(s)\big)\;ds\right)
\le 2\,\deltaK\;.
$$
If $0\le u\le \etaK(A)-2\,\deltaK$ we have
\begin{align*}
&\esperance_{\vn}\left(
\un_{\tauex>2\,\etaK(A)-u}\;
\int_{0}^{\etaK(A)-u}\un_{K*A^{c}}\big(\vN(s)\big)\;ds\right)\\
&\le
\deltaK+ 
\esperance_{\vn}\left(
\un_{\tauex>2\,\etaK(A)-u}\;
\int_{\deltaK}^{\etaK(A)-u}\un_{K*A^{c}}\big(\vN(s)\big)\;ds\right)\\
&
\le \deltaK+ 
\esperance_{\vn}\Bigg(\esperance_{\vN(\,\deltaK)}\Bigg(
\un_{\tauex>2\,\etaK(A)-u-\,\deltaK}\times 
\left(\int_{0}^{\etaK(A)-u-\,\deltaK}\un_{K*A^{c}}
\big(\vN(s)\big)\;ds\right)\Bigg)
\\
&
\le \deltaK+ 2\,\esperance_{\laKqsd}\Bigg(
\un_{\tauex>2\,\etaK(A)-u-\,\deltaK}\times\int_{0}^{\etaK(A)-u-\,\deltaK}\un_{K*A^{c}}
\big(\vN(s)\big)\;ds\Bigg)
+\lambK\,\etaK(A)
\end{align*}
for $K$ large enough by Corollary \ref{mieuxcor} (ii) and
\eqref{supuK}.

Using Lemma \ref{tmpAc} we get (since
$2\,\etaK(A)-u-\,\deltaK>\,\deltaK$)
$$
\esperance_{\laKqsd}\left(
\un_{\tauex>2\,\etaK(A)-u-\,\deltaK}\;
\int_{0}^{\etaK(A)-u-\,\deltaK}\un_{K*A^{c}}\big(\vN(s)\big)\;ds\right)
$$
$$
\le \etaK(A)\;C_{3} \, \big(
(\lambK\,\etaK(A)+1)\;\mu_{\sK}\big(K*A^{c}\big)\;+
\lambK\,\deltaK)\big)\;.
$$
Therefore we get for $0\le u\le \etaK(A)$
$$
\sup_{\vn\in\partial_{e}(K*A)}\esperance_{\vn}\left(
\un_{\tauex>(\MK(A)-j)\,\etaK(A)-u}\;
\int_{0}^{\etaK(A)-u}\un_{K*A^{c}}\big(\vN(s)\big)\;ds\right)
$$
$$
\le \;\deltaK+
 \etaK(A)\;C_{3}\,\big(
(\lambK\,\etaK(A)+1)\;\mu_{\sK}\big(K*A^{c}\big)+
\lambK\,\deltaK \big)\;.
$$
Hence
\begin{align*}
&\esperance_{\laKqsd}\left(\un_{X_{0}=1}\;
\un_{\tauex>(\MK(A)-j)\,\etaK(A)}\;
\int_{0}^{\etaK(A)}\un_{K*A^{c}}\big(\vN(s)\big)\;ds\right)
\\&\le \proba_{\laKqsd}(X_{0}=1)\;\big[
\deltaK+
 \etaK(A)\;C_{3}\,\big(
(\lambK\,\etaK(A)+1)\;\mu_{\sK}\big(K*A^{c}\big)+
\lambK\,\deltaK \big)
\big]
\\&
\le \frac{1}{\MK(A)}\;\big[
\;\deltaK+
 \etaK(A)\;C_{3}\,\big(
(\lambK\,\etaK(A)+1)\;\mu_{\sK}\big(K*A^{c}\big)+
\lambK\,\deltaK \big)
\big]
\end{align*}
by Lemma  \ref{encadre}.

Summing over $j$ we get 
\begin{align*}
&\esperance_{\laKqsd}\left(
\un_{\tauex>\MK(A)\,\etaK(A)}\;
\int_{0}^{\MK(A)\,\etaK(A)}\un_{K*A^{c}}\big(\vN(s)\big)\;ds\right)\\
&\le 
\;\deltaK+
C  \etaK(A)\;\big(
(\lambK\,\etaK(A)+1+\,\deltaK/\etaK(A))\;\mu_{\sK}\big(K*A^{c}\big)+
\lambK\,\deltaK\big)\;,
\end{align*}
for $C$ a suitable constant.
Using Lemma \ref{tmpAc} with $L=\MK(A)\,\etaK(A)$ we get for $K$ large
enough
that
$$
\mu_{K}(K*A^{c})\le\frac{2\, \;\deltaK}{\MK(A)\,\etaK(A)}\;.
$$
The first part of the Theorem is established. \\

\noindent The second part follows  from Theorem \ref{equivtout} and \eqref{supuK}.

\end{proof}

\section{Some global aspects of the trajectories.}
\label{sec:5}

In this section we derive some properties of the trajectories on larger
time scales.

\subsection{Extension of Theorem \ref{MainBD}.}

We can extend the result of Theorem \ref{MainBD} starting the
observation at a positive (diverging) time.

\begin{corollary}\label{decaldeb}
Let 
$A$    be  a subset of $\rpd$ satisfying Assumption {\bf (HS)}. 
Let $L(\sbullet)$ be a positive function such that
$$
\limsup_{K\to\infty}\lambK\;L(K)=0\;.
$$
Then
under the hypotheses {\bf (HV)}, 
with our previous choice of the integers 
$\etaK(A)$ and  $\MK(A)$,  for any fixed $s>0$ 
$$
\lim_{K\to\infty}d_{\mathrm{TV}}\left(
\mathscr{L}^{(\proba_{\laKqsd})}
\bigg(\sum_{j=0}^{[\MK(A)\,s]}X_{j}(A,K)\circ\theta_{L(K)}\bigg)
\,,\,\mathrm{Poisson}(s)\right)=0\;.
$$
\end{corollary}
\begin{proof}
We first note that using Theorem \ref{MainBDq} and the  time invariance of the q-process,
we have immediately that 
$$
\lim_{K\to\infty}d_{\mathrm{TV}}\left(
\mathscr{L}^{(\QK_{\mu_{\sK}})}
\bigg(\sum_{j=0}^{[\MK(A)\,s]}X_{j}(A,K)\circ\theta_{L(K)}\bigg)
\,,\,\mathrm{Poisson}(s)\right)=0\;.
$$

The corollary follows from Lemma \ref{connect} formulae \eqref{distancepq} and \eqref{distancetv} with 
$$
t(K) = ([\MK(A)\,s]+1)\etaK + L(K)
$$
using the triangle
inequality for the variation distance.

 \end{proof}

\subsection{Beyond the Poisson regime.}

Let $A$    be  a subset of $\rpd$ satisfying Assumption {\bf (HS)}. 
We consider a time scale much larger than
$\etaK(A)\;\MK(A) $. Intuitively, a Poisson distribution with large
parameter behaves like a Gaussian distribution centered on the
parameter and with variance equal to the parameter. 

We derive a similar result in our setting for a time scale slightly
larger than the average exit time of $K*A$.

\begin{theorem}\label{CLTgt}
Let $A$  be an open subset of $\rpd$ satisfying Assumption {\bf (HS)}.
There exists a constant $C_{A}>0$ such that for any
$$
\alpha\in (V_{*}(A),C_{A}+V_{*}(A))\;,
$$
 the random variable
$$
\frac{\sum_{j=0}^{[\exp(\alpha \,K)/\etaK(A)]}X_{j}(A,\,K)
-e^{\alpha \, K}\MK(A)^{-1}\etaK(A)^{-1}}{
e^{\alpha \, K/2}\MK(A)^{-1/2}\etaK(A)^{-1/2}}
$$
converges in law (for the distribution $\proba_{\laKqsd}$), as $K$
tends to infinity, to a standard normal random variable. 
\end{theorem}
\begin{proof}
From Theorem \ref{equivtout}, the
 assumption   $\alpha > V_{*}(A)$ implies that the quantity  
 $$
\exp(\alpha \,K)/\etaK(A)\MK(A)
$$
 tends to infinity with $K$.  The proof follows from a careful bookkeeping of the estimates  \eqref{eq:formula} in the proof of Theorem \ref{MainBD}. 
We then apply Lemma \ref{connect} formula \eqref{distancetv} and  the  convergence in law  of a suitably normalized Poisson
distribution to a normal distribution when the parameter diverges, see
for example \cite{handbookPoisson}.  
\end{proof}

In order to investigate longer times scales, we will look at the total
amount of time a trajectory has spent outside a set $A$ on a given 
time interval. 

\begin{theorem}\label{gdechelle}
Let $A$    be  a subset of $\rpd$ satisfying Assumption {\bf (HS)}.  Let 
$$
V_{*}(A)<\alpha<\alpha_{*}=-\limsup_{K\to\infty}\frac{\log(\lambK)}{K}\;.
$$
Let $\mathcal{T}(A,\,\alpha,\, K,\,(\vN(\vsbullet)))$ be the amount time a 
trajectory spends outside $A$ in the time interval 
$\big[0,\,\exp(K\,\alpha) \big]$, namely
$$
\mathcal{T}(A,\,\alpha,\,
K,\,(\vN(\vsbullet)))=\int_{0}^{\exp(K\,\alpha) }
\un_{K*A^{c}}\big(\vN(s)\big)\,ds\;.
$$
Then for any $\epsilon>0$
$$
\lim_{K\to\infty}\proba_{\laKqsd}\left(\frac{\mathcal{T}(A,\,\alpha,\,
 K,\,(\vN(\vsbullet)))}{e^{\alpha\,K}}
\in\big[(1-\epsilon)\,\mu_{\sK}\big(K*A^{c}\big),\,
(1+\epsilon)\,\mu_{\sK}\big(K*A^{c}\big)\big]
\right)=1.
$$
\end{theorem}

\begin{proof}
Let $\epsilon>0$ be fixed.

Noting first that since
$$\proba_{\laKqsd}\left(\tauex\le e^{K\,\alpha}\right)\le \lambK\; e^{K\,\alpha}\ ,$$
which goes to $0$ by assumption, 
it is enough to prove that 
$$\proba_{\laKqsd}\left(\left| \frac{\mathcal{T}(A,\,\alpha,\,
 K,\,(\vN(\vsbullet)))}{e^{\alpha\,K}}-\mu_{\sK}\big(K*A^{c}\big)\right|
  >\epsilon\,\mu_{\sK}\big(K*A^{c}\big),\,\tauex>e^{K\,\alpha}\right)
$$
 tends to $0$ as $K$ tends to infinity.

It follows from Lemma \ref{connect} formula \eqref{distancepq} that it
is enough to prove that
$$
\QK_{\mu_{\sK}}\left(\left| \frac{\mathcal{T}(A,\,\alpha,\,
 K,\,(\vN(\vsbullet)))}{e^{\alpha\,K}}-\mu_{\sK}\big(K*A^{c}\big)\right|
  >\epsilon\,\mu_{\sK}\big(K*A^{c}\big)\right)
$$
tends to $0$ as $K$ tends to infinity. 

Note that this differs from the ergodic theorem in the sense that we
consider only the finite time interval $[0,\exp(K\,\alpha)]$.

Let
$$
\moyenne=\EQK_{\mu_{\sK}}\left(\mathcal{T}(A,\,\alpha,\,
 K,\,(\vN(\vsbullet)))\right).
$$
By the time invariance of $\QK_{\mu_{\sK}}$ and Theorem \ref{gdvAc} we have
\begin{equation}\label{moyenne}
\moyenne=e^{\sK\,\alpha}\mu_{\sK}\big(K*A^{c}\big)= e^{K\,(\alpha-V_{*}(A)+o(1))}\;.
\end{equation}

The result will follow from Chebyshev's inequality if we can show that
\begin{equation}\label{desir}
\lim_{K\to\infty} \frac{\EQK_{\mu_{\sK}}\big(\mathcal{T}(A,\,\alpha,\,
 K,\,(\vN(\vsbullet)))^{2}\big)-\moyenne^{2}}{\moyenne^{2}}=0\,.
\end{equation}

For the second moment we have
\begin{align*}
&\EQK_{\mu_{\sK}}\left(\mathcal{T}(A,\,\alpha,\,
 K,\,(\vN(\vsbullet)))^{2}\right)
=2\int_{0\le s_{1}\le s_{2}\le \exp(\sK\,\alpha) }
\EQK_{\mu_{\sK}}\big(
\un_{K*A^{c}}\big(\vN(s_{1})\big)\;\un_{K*A^{c}}\big(\vN(s_{2})\big)
\big)\,ds_{1}\,ds_{2}\;.
\end{align*}
We will split the double integral into two terms in order to be able to
use the decorrelation. 

We have
$$
\int_{0\le s_{1}\le s_{2}\le \exp(\sK\,\alpha) }
\EQK_{\mu_{\sK}}\big(
\un_{K*A^{c}}\big(\vN(s_{1})\big)\;\un_{K*A^{c}}\big(\vN(s_{2})\big)\big)
\,ds_{1}\,ds_{2}=I_{1}+I_{2}
$$
with
\begin{align*}
I_{2}&=\int_{0\le s_{1}\le s_{2}\le (s_{1}+\,\deltaK)\wedge \exp(\sK\,\alpha) }
\EQK_{\mu_{\sK}}\big(
\un_{K*A^{c}}\big(\vN(s_{1})\big)\;\un_{K*A^{c}}\big(\vN(s_{2})\big)
\big)\,ds_{1}\,ds_{2}\\
& \le \deltaK\; \int_{0}^{\exp(\sK\,\alpha) }\EQK_{\mu_{\sK}}\big(
\un_{K*A^{c}}\big(\vN(s_{1})\big)\;\big)\;ds_{1}
=  \deltaK\;\moyenne\;,
\end{align*}
using \eqref{moyenne}.

In the term $I_{1}$ we use the Markov property and get
\begin{align*}
&I_{1}=\int_{ s_{1}\ge 0 \,;\, s_{1}+\,\deltaK\le s_{2}\le
  \exp(\sK\,\alpha) }
\EQK_{\mu_{\sK}}\bigg(
\un_{K*A^{c}}\big(\vN(s_{1})\big)
\;\esperance_{\vN(s_{1}+\,\deltaK)}\big(
\un_{K*A^{c}}\big(\vN(s_{2}-s_{1}-\,\deltaK)\big)\big)\bigg)\,ds_{1}\,ds_{2}\;.
\end{align*}

Using Theorem \ref{thm:mixing} and \eqref{moyenne}  we have
$$
I_{1}=I_{3}+I_{3}'
$$
with 
$$
I_{3}'\le \moyenne\;e^{\alpha\sK} \;C'_{1}\,e^{-\deltaK\;C'_{3}/\log K}.
$$

We have using the time invariance of $\QK_{\mu_{\sK}}$
\begin{align*}
I_{3}&=\int_{ s_{1}\ge 0 \,;\, s_{1}+\,\deltaK\le s_{2}\le
  \exp(\sK\,\alpha) }
\EQK_{\mu_{\sK}}\big(
\un_{K*A^{c}}\big(\vN(s_{1})\big) 
\;\EQK_{\mu_{\sK}}\big(
\un_{K*A^{c}}\big(\vN(s_{2}-s_{1}-\,\deltaK)\big)\big)\big)\,ds_{1}\,ds_{2}\;\\
&=\mu_{\sK}\big(K*A^{c}\big)^{2}\;
 \int_{ s_{1}\ge 0 \,;\, s_{1}+\,\deltaK\le s_{2}\le
  \exp(\sK\,\alpha)
}ds_{1}\,ds_{2}\\
&=\frac{1}{2}\;\mu_{\sK}\big(K*A^{c}\big)^{2}
\left(e^{\sK\,\alpha}-\deltaK\right)^{2}\\
&=\frac{1}{2}\,\moyenne^{2}-\deltaK\, \mu_{\sK}\big(K*A^{c}\big)
\moyenne+\frac{1}{2} \mu_{\sK}\big(K*A^{c}\big)^{2}\,\deltaK^{2}\;.
\end{align*}

Summarizing we have  obtained the upper bound
\begin{align*}
&\EQK_{\mu_{\sK}}\left(\mathcal{T}(A,\,\alpha,\,
 K,\,(\vN(\vsbullet)))^{2}\un_{\tauex>e^{K\,\alpha}}\right)-\moyenne^{2}\\
\le &
 2 \deltaK\, \moyenne+ 2\moyenne\;e^{\alpha\sK}
      \;C'_{1}\,e^{-\deltaK\;C'_{3}/\log K}
+\mu_{\sK}\big(K*A^{c}\big)^{2}\,\deltaK^{2}.
\end{align*}
We conclude using the condition on  $\alpha$  and \eqref{choix:deltaK} that
$$
\lim_{K\to\infty}\frac{\esperance_{\laKqsd}\left(\mathcal{T}(A,\,\alpha,\,
 K,\,(\vN(\vsbullet)))^{2}\un_{\tauex>e^{K\,\alpha}}\right)-\moyenne^{2}
}{\moyenne^{2} }=0\;.
$$
As explained before the result follows using Chebyshev's inequality.
\end{proof}

\subsection{Asymptotic record profile}

For simplicity we  consider only the case of dimension one,
similar arguments can be developed for higher dimensions.

\begin{figure}[!t]
\begin{center}
    \includegraphics[width=10cm]{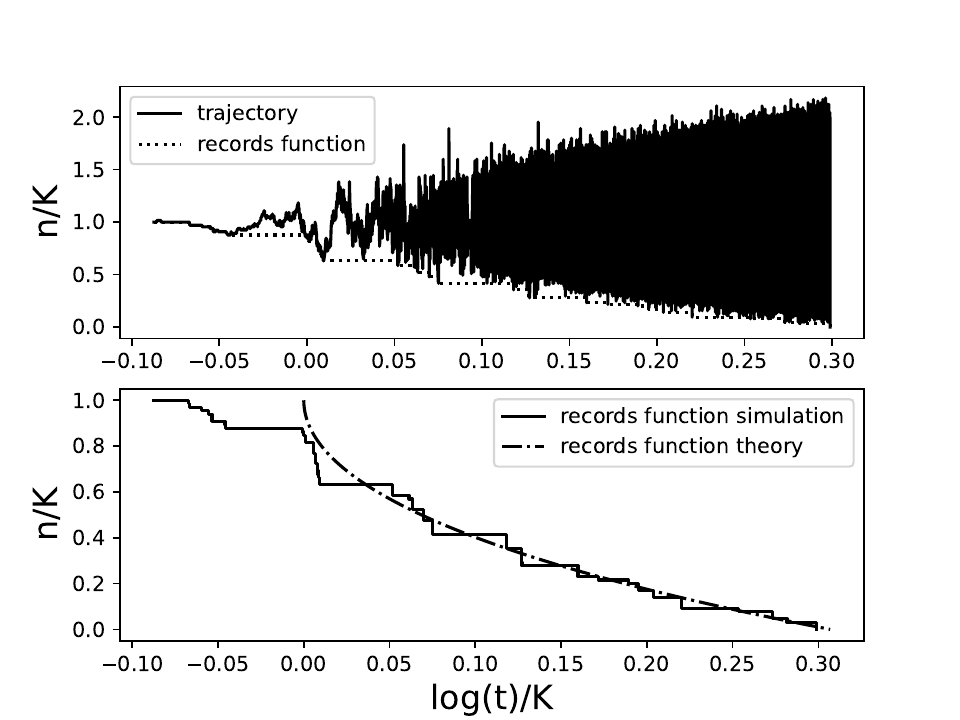}
\end{center}
    \caption{A simulation for a one specie model with $K=65$,
$B(x)=2\,x$ and $D(x)=x+x^{2}$ and initial condition $\nf$.
 Top: trajectory of
      the process up to extinction.   Bottom :  records as functions
      of time. The values of $K$ used in the simulations are limited because
computation time grows exponentially fast with $K$.
}
    \label{figtout}
\end{figure}

\medskip
Given a trajectory $(\NK(\vsbullet))$ of the process, we define the record process for $s\ge0$ by
$$
r_{K}\big(s,\, \big(\NK(\vsbullet)\big)\big)=\inf_{0\le v\le
  \exp(K\,s)-1}\frac{\NK(v)}{K}\;. 
$$
For fixed $K$ this is a random function of $s$.

Let us introduce the (lower) record profile $r^{*}(\vsbullet)$, defined  for $0\le s\le V_{*}\big((0,\,\infty)\big)$ 
by 
$$
r^{*}(s)=\min\big\{z\le \xf,\; s=V_{*}\big((z,\,\infty)\big)\big\}\;.
$$
In dimension one, due to the explicit expression in  \ref{1dformule}
this is the inverse of the continuous strictly decreasing function
$$
V_{*}\big((u,\,\infty)\big)=\int_{u}^{\xf}\log\big(B(y)/D(y)\big)\;dy\;.
$$
Therefore the function $r^{*}$
 is strictly decreasing and continuous.
 Note that $V_{*}\big((\xf,\,\infty)\big) = 0$, hence  $r^{*}(0)=\xf$, and also $r^{*}(V_{*}\big((0,\,\infty)\big)) = 0$.

\begin{theorem}\label{conv:rec}
When the process $(\NK)$ starts  from $\laKqsd$, 
the random variable 
$$
\sup_{0\le s\le V_{*}((0,\,\infty))}\big|r_{K}\big(s,\, 
\big(\NK(\vsbullet)\big)\big)-r^{*}(s)\big|
$$
converges to zero in probability when $K$ tends to infinity.
\end{theorem}
See Figure \ref{figtout}-Bottom for a numerical illustration of this result. \\

We first prove a preliminary lemma.

\begin{lemma}\label{zparz}
When the process $(\NK)$ starts  from $\laKqsd$, 
for each $z\in[0,\,V_{*}((0,\,\infty))]$ the random variable 
$r_{K}\big(z,\, \big(\NK(\vsbullet)\big)\big)$ converges to $r^{*}(z)$
in probability when $K$ tends to infinity. In other words, for any $\rho>0$
$$
\lim_{K\to\infty}\proba_{\laKqsd}\left(\big|r_{K}\big(z,\,
  \big(\NK(\vsbullet)\big)\big)-r^{*}(z)\big|>\rho\right)=0\;.
$$
\end{lemma}

\begin{proof}

It is enough to assume $0<\rho<\xf/2$.

We will consider separately the two cases 
$0\le z<V_{*}((0,\,\infty))$ and $z=V_{*}((0,\,\infty))$.

In the case $0\le z<V_{*}((0,\,\infty))$
we have
\begin{align*}
&\proba_{\laKqsd}\left(\big|r_{K}\big(z,\,
  \big(\NK(\vsbullet)\big)\big)-r^{*}(z)\big|>\rho\right)
\\&=\proba_{\laKqsd}\left(r_{K}\big(z,\,\big(\NK(\vsbullet)\big)\big)
>r^{*}(z)+\rho\right)
+\proba_{\laKqsd}\left(r_{K}\big(z,\,\big(\NK(\vsbullet)\big)\big)
<r^{*}(z)-\rho\right),
\end{align*}
and we will show that these two numbers tend to zero when $K$ tends to
infinity. 

For the first term, let $V_{*}((0,\,\infty))>\zeta>0$ be defined by
$r^{*}(\zeta)=\xf-\rho/2$. If $0\le z\le \zeta$ we have by the
monotonicity of $r^{*}(z)$ and $r_{K}\big(z,\,\big(\NK(\vsbullet)\big)$ in $z$, that
$$
\proba_{\laKqsd}\left(r_{K}\big(z,\,\big(\NK(\vsbullet)\big)\big)
>r^{*}(z)+\rho\right)\le 
\proba_{\laKqsd}\left(r_{K}\big(z,\,\big(\NK(\vsbullet)\big)\big)
>\xf+\rho/2\right)
$$
$$
\le \proba_{\laKqsd}\left(r_{K}\big(0,\,\big(\NK(\vsbullet)\big)\big)
>\xf+\rho/2\right)=\laKqsd\left(\NK(0)>K\xf+K\,\rho/2\right)\;.
$$
 
From Theorem 3.7 in \cite{CCM1}, denoting $\nf=[K\xf]$,  we have 
$$
\laKqsd\big(n-\nf>[K^{2/3}])\le \frac{\Oun}{\sqrt{K}}+\Oun
\frac{\sum_{p>\nf+[K^{2/3}]}^{\infty}
e^{-(p-\nf)^{2}}/(2\,K\,\sigma^{2})}{\sqrt{K}}\;,
$$ 
with 
$$
\sigma=\left(\frac{d}{dx}\log\frac{D(x_{*})}{B(x_{*})}\right)^{-1/2}<\infty\;.
$$
This implies
$$
\lim_{\sK\to\infty} \laKqsd\big(n>\nf+[K^{2/3}])=0\;,
$$
which implies
$$
\lim_{\sK\to\infty}\laKqsd\left(\NK(0)>K\,\xf+K\,\rho/2\right)=0\;,
$$
since for $K$ large enough $K\,\rho/2>[K^{2/3}] $.

For $\zeta < z < V_{*}((0,\,\infty))$
we observe that
$$
\big\{r_{K}\big(z,\,\vsbullet\big)>r^{*}(z)+\rho\big\}=
\left\{\NK(s)> K\,(r^{*}(z)+\rho),\; 0\le s\le e^{K\,z}-1\right\}\;.
$$
We have $r^{*}(z)<r^{*}(\zeta)<\xf$, hence we can find $z'<z$ such that 
$r^{*}(z)<r^{*}(z')<r^{*}(z)+\rho$ and $r^{*}(z')<\xf$. We have
\begin{align*}
\left\{\NK(s)> K\,(r^{*}(z)+\rho),\; 0\le s\le e^{K\,z}-1\right\}
&\subset \left\{\NK(s)> K\,r^{*}(z'),\; 0\le s\le e^{K\,z}-1\right\}
\\&
=\left\{\TK_{(r^{*}(z'),\,\infty)}>e^{K\,z}-1\right\}\;.
\end{align*}
It follows from  Theorem \ref{echelleTAK} that
$$
\lim_{K\to\infty}\proba_{\laKqsd}\left(r_{K}\big(z,\,\big(\NK(\vsbullet)\big)\big)
>r^{*}(z)+\rho\right)\le \lim_{K\to\infty}
\proba_{\laKqsd}\left(\TK_{(r^{*}(z'),\,\infty)}>e^{K\,z}-1\right)=0
$$
since $z>V_{*}((r^{*}(z'),\,\infty))=z'$.
This shows that for all $z\neq V_{*}((0,\,\infty))$,
$$\lim_{K\to\infty}\proba_{\laKqsd}\left(r_{K}\big(z,\,\big(\NK(\vsbullet)\big)\big)
>r^{*}(z)+\rho\right)=0.$$

\medskip For the second term, since
$r_{K}\big(z,\,\big(\NK(\vsbullet)\big)\ge0$, we only have to consider
$r^{*}(z)-\rho>0$, namely 
$z< z_{\rho}$ with $r^{*}(z_{\rho})=\rho$ which implies
$z_{\rho}<V_{*}((0,\,\infty))$. 

We observe that
\begin{align*}
\big\{r_{K}\big(z,\,\vsbullet\big)<r^{*}(z)-\rho\big\}&=
\left\{\exists\; 0\le s\le e^{K\,z}-1,\;
\NK(s)< K\,(r^{*}(z)-\rho)\right\}
\\&
=\left\{\TK_{(r^{*}(z)-\rho,\,\infty)}<e^{K\,z}-1\right\}\;.
\end{align*}

Since $z<z_{\theta}$ we have $V_{*}((r^{*}(z)-\rho,\,\infty))>z$ 
and by Theorem \ref{echelleTAK}
$$
\lim_{K\to\infty}\proba_{\laKqsd}\left(r_{K}\big(z,\,\big(\NK(\vsbullet)\big)\big)
<r^{*}(z)-\rho\right)=0\;.
$$
This implies that the second term vanishes for $z\neq V_{*}((0,\,\infty))$ when $K$ tends to infinity.

\medskip
For the case $z=V_{*}((0,\,\infty))$ we have
\begin{align*}
\proba_{\laKqsd}\left(\big|r_{K}\big(z,\,
  \big(\NK(\vsbullet)\big)\big)-r^{*}(z)\big|>\rho\right)
&=\proba_{\laKqsd}\left(r_{K}\big(V_{*}((0,\,\infty)),\,
  \big(\NK(\vsbullet)\big)\big)>\rho\right)
\\&=\proba_{\laKqsd}\left(\NK(s)> K\,\rho,\; 0\le s\le e^{K\,V_{*}((0,\,\infty))}-1
\right)
\\&
=\proba_{\laKqsd}\left(\TK_{(\rho,\infty)}>e^{K\,V_{*}((0,\,\infty))}-1\right)
\end{align*}
which tends to $0$ when $K$ tends to infinity from Theorem
\ref{echelleTAK} since $V_{*}((0,\,\infty))>V_{*}((\rho,\,\infty))$.

\end{proof}

\begin{proof}[Proof of Theorem \ref{conv:rec}]

Since the interval  $[0,\,V_{*}((0,\infty))]$ is compact, the function
$r^{*}(s)$ is uniformly continuous on this interval. 
Fix $\rho>0$, we can find a
finite increasing sequence $0=z_{1}<\ldots<z_{p}=V_{*}((0,\infty))$ of points
such that for any $1\le q\le p-1$
$$
\big|r^{*}(z_{q+1})-r^{*}(z_{q})\big|<\frac{\rho}{2}\;.
$$
In particular, for each $z\in [0,\,V_{*}((0,\infty))]$,
 there exists $1\le q\le p-1$
such that $z\in [z_{q},\,z_{q+1}]$, hence
$$
r^{*}(z_{q})\ge r^{*}(z)\ge r^{*}(z_{q+1})\;.
$$
Since $r_{K}\big(s,\, \big(\NK(\vsbullet)\big)$
is nonincreasing in $s$ we have for such a $z$ and $q$
\begin{align*}
 r_{K}\big(z,\, \big(\NK(\vsbullet)\big)- r^{*}(z)&\le 
r_{K}\big(z_{q},\, \big(\NK(\vsbullet)\big)-
r^{*}(z_{q})+r^{*}(z_{q})-r^{*}(z)
\\ &\le r_{K}\big(z_{q},\, \big(\NK(\vsbullet)\big)-
r^{*}(z_{q})+ \frac{\rho}{2}\;.
\end{align*}
Similarly
\begin{align*}
r_{K}\big(z,\, \big(\NK(\vsbullet)\big)- r^{*}(z)&\ge 
r_{K}\big(z_{q+1},\, \big(\NK(\vsbullet)\big)-
r^{*}(z_{q+1})+r^{*}(z_{q+1})-r^{*}(z)
\\
&\ge r_{K}\big(z_{q+1},\, \big(\NK(\vsbullet)\big)-
r^{*}(z_{q+1})-\frac{\rho}{2}\;.
\end{align*}
Therefore, for any $1\le q\le p-1$
\begin{align*}
\proba_{\laKqsd}
\left(\sup_{z_{q}\le z\le z_{q+1}}\big|r_{K}\big(z,\,
\big(\NK(\vsbullet)\big)\big)-r^{*}(z)\big|>
\rho/2\right)
&\le \proba_{\laKqsd}
\left(\big|r_{K}\big(z_{q},\,
\big(\NK(\vsbullet)\big)\big)-r^{*}(z_{q})\big|>
\rho\right)
\\&
\hskip 0.5cm +\proba_{\laKqsd}
\left(\big|r_{K}\big(z_{q+1},\,
\big(\NK(\vsbullet)\big)\big)-r^{*}(z_{q+1})\big|>
\rho\right)\;.
\end{align*}
This implies
$$
\proba_{\laKqsd}
\left(\sup_{0\le s\le V_{*}((0,\,\infty))}\big|r_{K}\big(s,\, 
\big(\NK(\vsbullet)\big)\big)-r^{*}(s)\big|>\rho/2\right)
\le 2 \sum_{j=1}^{p} \proba_{\laKqsd}
\left(\big|r_{K}\big(z_{j},\,
\big(\NK(\vsbullet)\big)\big)-r^{*}(z_{j})\big|>
\rho\right)
$$
which tends to $0$ when $K$ tends to infinity by Lemma \ref{zparz}.
The result follows. 

\end{proof}

In words, if one draws with a finite precision the graph of $\NK(t)$
as a function of $\log(t)/K$, for large $K$, the picture is completely
black above the graph of the record profile $r^{*}$ at least until
$\xf$ (up 
to the other upper branch of $V_{*}^{-1}\big(\vsbullet,\infty)$) 
while the picture is
completely white below.

\appendix

\section{Proof of theorem \ref{ordreV}.}\label{appendice1}

For the convenience of the reader we recall the setting and the statement.

\begin{theorem}

Let $A$ and $D$ open subsets of $\rpd$ satisfying Assumption ${\bf (HS)}
$ and $ A\Subset D$. 

Let $V_{*}(A)$ (respectively $V_{*}(D)$) denote the large deviation
potential for $A$ (respectively $D$),  we have
$$
0<V_{*}(A)<V_{*}(D)\;.
$$
\end{theorem}
Although this result looks natural, the strict inequalities requires a
proof.

\begin{remark}\label{1dformule}
For $d=1$ the result follows immediately from the explicit expression
of the large deviations potential. Namely, if $0\le a<\xf<b$ 
\begin{equation}
V_{*}\big((a,b)\big)=\min\left\{\int_{a}^{\xf}\log\big(B(y)/D(y)\big)\,dy\,,\,
\int_{\xf}^{b}\log\big(D(y)/B(y)\big)\,\,dy\right\}\;,
\end{equation}
see \cite{VF} page 140.
\end{remark}

\begin{proof}

 Recall that the large deviation functional is given by 
$$
I_{0}^{T}(\vr)=\int_{0}^{T}\ell\big(\vr(s),\,{\vr}'(s)\big)\;ds,
$$
if $\vr(s)$ is absolutely continuous, otherwise infinite, where $\ell$ is the Lagragian (or local rate function) defined in \cite{SW} p. 70.

For any $\epsilon>0$, there exists a time $T_{\epsilon}>0$ and a path
$\vr_{\epsilon}$ with $\vr_{\epsilon}([0,T_{\epsilon})) \subset D$, 
absolutely continuous such that $\vr_{\epsilon}(0)=\xf$,  $\vr_{\epsilon}(T_{\epsilon})\in \partial D$ and 
$$
V_{*}(D)\le I_{0}^{T_{\epsilon}}(\vr_{\epsilon})\le V_{*}(D)+\epsilon\;.
$$
Being absolutely continuous, $\vr_{\epsilon}$ is continuous.

We will choose $\epsilon$ adequately later on, namely $0<
\epsilon\le \epsilon_{1}\wedge\epsilon_{2}$ where $\epsilon_{1}$ and
$\epsilon_{2}$ are two positive constants depending only on $A$, $D$ and
the vector field.

Let $T'_{\epsilon}$ and $T''_{\epsilon}$ be defined by
\begin{align*}
T'_{\epsilon}&=\inf\{t>0\,,\, \vr_{\epsilon}(t)\in\partial A\}\;,
\\
T''_{\epsilon}&=\sup\{t<T_{\epsilon}\,,\, \vr_{\epsilon}(t)\in\partial A\}\;.
\end{align*}
We have obviously since $\ell\ge0$ (see Proposition 5.10 page 85 in \cite{SW}) that 
$$
I_{0}^{T_{\epsilon}}(\vr_{\epsilon})\ge
I_{0}^{T'_{\epsilon}}(\vr_{\epsilon})
+I_{T''_{\epsilon}}^{T_{\epsilon}}(\vr_{\epsilon})\ge V_{*}(A)+
I_{T''_{\epsilon}}^{T_{\epsilon}}(\vr_{\epsilon})\;.
$$
If 
$\ 
I_{T''_{\epsilon}}^{T_{\epsilon}}(\vr_{\epsilon})> \epsilon
\ $
the result follows. 

Assume now  that
$$
I_{T''_{\epsilon}}^{T_{\epsilon}}(\vr_{\epsilon})\le \epsilon\;.
$$
We will show that for $\epsilon$ small enough this leads to a
contradiction.

Let 
$$
{\vy}_{\epsilon}=\vr_{\epsilon}(T''_{\epsilon})\;,\qquad
 {\vz}_{\epsilon}=\vr_{\epsilon}(T_{\epsilon})\;.
$$
Since $D$ is bounded, it follows from Lemma 5.17 page 87 in \cite{SW}
that there exist constants $C_{1}$ and $B_{1}>1$ independent of
$\vr_{\epsilon}$ such that
\begin{align*}
I_{T''_{\epsilon}}^{T_{\epsilon}}(\vr_{\epsilon})&\ge
C_{1}\int_{[T''_{\epsilon},\,T_{\epsilon}]\cap\{\|{\vr_{\epsilon}}'(\,\sbullet\,)\|>B_{1}\}}
 \|{\vr}'_{\epsilon}(s)\|\;\log\|{\vr_{\epsilon}}'(s)\|\;ds
\\&
\ge C_{1}\;\log
B_{1}\;\int_{[T''_{\epsilon},\,T_{\epsilon}]\cap\{\|{\vr_{\epsilon}}'(s)\|>B_{1}\}}
 \|{\vr}'_{\epsilon}(s)\|\;ds\;.
\end{align*}
Let 
$$
\delta=\inf_{\vy\in\partial A,\, \vz\in\partial D}\|\vy-\vz\|>0\;.
$$
We have
\begin{align*}
\delta\le \|{\vy}_{\epsilon}-{\vz}_{\epsilon}\|&=
\left\|\int_{T''_{\epsilon}}^{T_{\epsilon}}\vr'_{\epsilon}(s)\;ds\right\|
\\
&
\le \int_{[T''_{\epsilon},\,T_{\epsilon}]\cap\{\|{\vr}'_{\epsilon}(\,\sbullet\,)\|
  \le B_{1}\}} \|\vr'_{\epsilon}(s)\|\;ds +
\int_{[T''_{\epsilon},\,T_{\epsilon}]\cap\{\|\vr'_{\epsilon}(\,\sbullet\,)\|
 > B_{1}\}}\|\vr'_{\epsilon}(s)\|\;ds
\\&
\le B_{1}\;\big(T_{\epsilon}-T''_{\epsilon}\big)
+\frac{I_{T''_{\epsilon}}^{T_{\epsilon}}(\vr_{\epsilon})}{C_{1}\;\log
  B_{1}}
\le  B_{1}\;\big(T_{\epsilon}-T''_{\epsilon}\big)
+\frac{\epsilon}{C_{1}\;\log B_{1}}\;.
\end{align*}
Take 
$$
\epsilon_{1}=C_{1}\;\frac{\delta}{2}\;\log B_{1}\;.
$$
We get
$$
T_{\epsilon}-T''_{\epsilon}\ge \frac{\delta}{2\,B_{1}}\;.
$$

We will now use the ideas of  the proof of  Lemme 6.28 page 140 in
\cite{SW}.

Let 
$$
T_{*}=\frac{\delta}{2\,B_{1}}\wedge 1\;.
$$
We have of course
$$
I_{T''_{\epsilon}}^{T_{\epsilon}}(\vr_{\epsilon})\ge
I_{T''_{\epsilon}}^{T_{*}+T''_{\epsilon}}(\vr_{\epsilon})\;.
$$

Let us consider  an open set $G$ satisfying Assumption ${\bf (HS)}$ and such that $D\Subset G$. (It is easy to verify that $G$ exists). 
We recall Theorem \ref{lipG} (ii).
There exist $C_{1}(D,G)$ and $C_{2}(D,G)$ and $\zeta(D,G)>0$ such that for any
$\zeta(D,G)>\zeta>0$ and $\vn$ such that $d(\frac{\vn}{K},D)\le  {1\over 2}  d(\partial D, \partial G)$,
\begin{equation}\label{blabla}
\proba_{\vn}\big(\sup_{0\le t\le
  T_{*}}\|\vN(t)-K\varphi_{t}(\vn/K)\|>\,K\zeta\big)\le
C_{1}(D,G)\; e^{-K\,C_{2}(D,G)\,\zeta}\;.
\end{equation}

Let 
$$
S_{T_{*},\,\zeta}=\big\{\vr\in \mathbb{D}^{d}[0,T_{*}]\,,\,\sup_{0\le s\le
  T_{*}}\|\vr( s)-\varphi_{s}(\vy_{\epsilon})\|>\zeta\,\big\} \;.
$$
It is left to the reader to check that 
$S_{T_{*},\,\zeta}$ is open in the Skorohod topology.

By the large deviations lower bound Theorem 5.51 page 107 in \cite{SW} we
have
$$
\proba_{K \vy_{\epsilon}}\big(\frac{1}{K}\vN_{\big|[0,T_{*}]}\in
S_{T_{*},\,\zeta}\big)\ge e^{-K\,I^{*}(S_{T_{*},\,\zeta}) +o(K)}
$$
where
$$
I^{*}(S_{T_{*},\,\zeta})=\inf\big\{I_{0}^{T_{*}}(\vr)\,,
\,\vr\in S_{T_{*},\,\zeta}, \, \vr(0)= \vy_{\epsilon} \big\}\;.
$$
Therefore using \eqref{blabla} for $0<\zeta<\zeta(D,G)$  we get 
$$
I^{*}(S_{T_{*},\,\zeta})\ge C_{2}(D,G)\,\zeta\;.
$$
Since $A$ is invariant by the flow, and the condition on the normal at
the boundary (Assumption ${\bf (HS)}$), we have that
$$
\zeta_{*}=\inf_{\vy\in\partial A}\sup_{0\le t\le T_{*}}d\big(\varphi_{t}(\vy),\,
\partial A)>0\;.
$$
For $\zeta_{**}=\zeta_{*}\wedge\zeta(D,G)>0$,
since $\vr_{\epsilon}(0)=\vy_{\epsilon}\in \partial A$ and
$\vr_{\epsilon}([T''_{\epsilon},\, T_{*}+T''_{\epsilon}])\subset A^{c}$, we have that $\vr_{\epsilon}(T''_{\epsilon}+\vsbullet) | [0,\,T_{*}]\in S_{T_{*},\,\zeta_{**}}$.
This implies 
$$
I_{T''_{\epsilon}}^{T_{*}}(\vr_{\epsilon})\ge I^*(S_{T_{*},\,\zeta_{**}})
\ge C_{2}\, \zeta_{**}>0\;.
$$
We choose
$$
\epsilon_{2}=C_{2}\,\zeta_{**}/2\;,
$$
and the contradiction follows.
\end{proof}

\section{}

\begin{lemma}\label{conv:uKnf}Let $\vx\in \rpd$. Assume there exists
  $A$,  an open subset of $\rpd$
  satisfying Assumption {\bf (HS)} such that $\vx\in A$. Then 
$$
\lim_{K\to\infty} \lauK([K\,\vx])=1\;.
$$
\end{lemma}
\begin{remark}
A stronger result was established  in dimension one in \cite{CCM1}
(see Remark 3.8).
\end{remark}

\begin{proof} Using Corollary \ref{iter:Kurtz} with $p=[K\,\log K]$ we get 
$$
\proba_{[K\,\vx]}\big( \vN([K\,\log K])\in K*A\big)\ge 
1-[K\,\log K]\,e^{-C(\vx,A)\,K} \;.
$$
Using \eqref{melange} we get

$$
e^{\lambK\,[K\,\log K]}
\frac{\proba_{[K\,\vx]}\big( \vN([K\,\log K])\in K*A\big)-\Gamma([K\,\log
K],K)}{\laKqsd(K*A)}\le \lauK([K\,\vx])\;.
$$
Since
$$
\laKqsd(K*A)\ge \laKqsd\big(\|\vn-\nf\|\le \sqrt{K}\,\log K\big)\, ,
$$
it follows from Chebyshev's inequality, Theorem 2.6 and Proposition
2.7 of  \cite{CCMM}
that
$$
1\ge \limsup_{K\to\infty}\laKqsd(K*A)\ge
\liminf_{K\to\infty}\laKqsd(K*A)
$$
$$
\ge\lim_{K\to\infty}
\laKqsd\big(\|\vn-\nf\|\le \sqrt{K}\,\log K\big)=1\,.
$$
The result follows using \eqref{supuK} and the previous estimates.
\end{proof}



\end{document}